\documentclass[a4paper,11pt]{amsart}
\usepackage[english]{babel}
\usepackage{amsmath,amscd,amssymb,amsthm,amsxtra,enumerate}
\usepackage{mathrsfs}
\usepackage{mathtools}
\usepackage{color,graphicx}
\usepackage{epsfig,epstopdf}
\usepackage[margin=1.25in]{geometry}
\usepackage{enumerate}
\usepackage{hyperref}

\newcommand{\norm}[1]{\left\Vert#1\right\Vert}

\newcommand{\R}{\mathbb{R}}
\newcommand{\C}{\mathbb{C}}

\newcommand{\blue}[1]{\textcolor{blue}{#1}}

\renewcommand{\Re}{\operatorname{Re}}

\newcommand\reallywidehat[1]{\arraycolsep=0pt\relax%
\begin{array}{c}
\stretchto{
  \scaleto{
    \scalerel*[\widthof{\ensuremath{#1}}]{\kern-.5pt\bigwedge\kern-.5pt}
    {\rule[-\textheight/2]{1ex}{\textheight}} 
  }{\textheight} %
}{0.5ex}\\           
#1\\                 
\rule{-1ex}{0ex}
\end{array}
}

\DeclareMathOperator{\dvie}{div}
\DeclareMathOperator{\Dom}{Dom}

\newtheorem{thm}{Theorem}[section]

\newtheorem{lem}[thm]{Lemma}

\theoremstyle{definition}

\newtheorem{rem}[thm]{Remark}

\numberwithin{equation}{section}

\allowdisplaybreaks

\author[A. Biswas]{Animesh Biswas}

\address{Department of Mathematics\\
University of Nebraska-Lincoln\\
210 Avery Hall, Lincoln\\
NE 68588, United States of America}
\email{abiswas2@unl.edu}

\author[M. De Le\'on-Contreras]{Marta De Le\'on-Contreras}

\address{Department of Mathematics and Statistics\\
University of Reading\\
Whiteknights, PO Box 220\\
Reading RG6 6AX, United Kingdom}
\email{m.deleoncontreras@reading.ac.uk}

\author[P. R. Stinga]{Pablo Ra\'ul Stinga}

\address{Department of Mathematics\\
Iowa State University\\
396 Carver Hall, Ames\\
IA 50011, United States of America}
\email{stinga@iastate.edu}

\thanks{Research supported by Simons Foundation grant 580911 (PRS)}



\begin{document}

\title[Master equations]{Harnack inequalities
and H\"older estimates \\ for master equations}

\begin{abstract}
We study master equations of the form
$$(\partial_t+L)^su=f\quad\hbox{in}~\R\times\Omega$$
where $L$ is a divergence form elliptic operator and $\Omega\subseteq\R^n$.
These are nonlocal equations of order $2s$ in space and $s$ in time
that take into account the values of $u$ everywhere in $\Omega$ and for past times.
We show parabolic interior and boundary Harnack inequalities and local parabolic H\"older continuity of
solutions. To this end, we prove a characterization of
fractional powers of parabolic operators $\partial_t+L$ with a degenerate parabolic extension problem.
\end{abstract}

\maketitle

\section{Introduction}

We study regularity estimates for master equations driven by fractional powers of parabolic operators of the form
\begin{equation}\label{eq:Hsequation}
H^s u(t,x)\equiv(\partial_t+L)^su(t,x)=f(t,x)\quad0<s<1
\end{equation}
for $t\in\R$ and $x\in\Omega$, where $\Omega$ is an open subset of $\R^n$, $n\geq1$,
that may be unbounded, and $L$ is an elliptic operator
subject to appropriate boundary conditions on $\partial\Omega$.

Master equations are of great interest in mechanics, elasticity, biology and physics.
Consider, for example, the semipermeable membrane problem in biology, which is also equivalent to the 
parabolic Signorini problem in elasticity. In this problem, we have an anisotropic diffusion process
happening inside a cell modeled as
a region of points $(x,y)\in\mathcal{R}=\Omega\times(0,\infty)$. The diffusion is driven by a parabolic equation
$\partial_tU=-LU+\partial_{yy}U$, where $U(t,x,y)$ is the pressure inside the cell
at time $t$ at the point $(x,y)\in \mathcal{R}$. Here $-L$ is an elliptic operator in the variable $x\in\Omega$.
For instance, $L$ can be the Laplacian $-\Delta$ in $\Omega$ or a divergence form operator as in \eqref{eq:operatorL}.
The bottom of the cell $\mathcal{T}=\Omega\times\{y=0\}$
is a semipermeable membrane, meaning that if the pressure $U(t,x,0)$ on the membrane
becomes smaller than the pressure $\phi(t,x)$ from outside, then
fluid can enter the cell, otherwise there is no flux. On the wall of the cell $\partial\Omega\times(0,\infty)$ one may
assume that pressure is zero, that is $U=0$, or there is no flux, namely, $ \partial_\nu U=0$,
where $\partial_\nu$ denotes the exterior normal derivative at the boundary $\partial \Omega$.
As in \cite{DL}, if we assume that the membrane $\mathcal{T}$ is very thin, then the model becomes
\begin{equation}\label{eq:semi_perm_L}
    \begin{cases}
\partial_t U =-LU +\partial_{yy} U&\hbox{for}~t\in\R,~(x,y)\in\mathcal{R}, \\
 U(t,x,0) \geq \phi(t,x)&\hbox{for}~t\in\R,~\hbox{on}~\mathcal{T}, \\ 
 -\partial_y U(t,x,0) \geq 0&\hbox{on}~\mathcal{T}, \\
 -\partial_y U(t,x,0) = 0 &\hbox{whenever}~U> \phi~\hbox{on}~\mathcal{T},\\
 U= 0 \quad \hbox{or}\quad \partial_\nu U= 0 &\hbox{for}~t\in\R,~\hbox{on the wall}~\partial \Omega \times (0, \infty).
\end{cases}
\end{equation}
Then, as done in \cite{Biswas}, it can be seen that the flux is given by
$$-\partial_y U(t,x,0) = (\partial_t+L)^{1/2} U(t,x,0)\qquad\hbox{on}~\mathcal{T}.$$
Moreover, the semipermeable membrane model \eqref{eq:semi_perm_L} turns out to be equivalent to the following
obstacle problem on the membrane for $u(t,x):=U(t,x,0)$ when $s=1/2$:
$$\begin{cases}
u\geq\phi&\hbox{for all}~(t,x)\in\R\times\Omega\\
(\partial_t+L)^su\geq0&\hbox{for}~(t,x)\in\R\times\Omega,\\
(\partial_t+L)^su=0&\hbox{whenever}~u>\phi,\\
u= 0 \quad \hbox{or}\quad \partial_\nu u= 0 &\hbox{for}~t\in\R,~\hbox{on}~\partial \Omega.
\end{cases}$$
This free boundary problem for $0<s<1$ when $L=-\Delta$ and $\Omega=\R^n$ was studied in \cite{ACM}.
It can also be seen that the problem of biological invasions where there is a road with fast diffusion
considered in \cite{BRR} is equivalent to a local-nonlocal system driven by a local parabolic equation and a nonlocal equation
as in \eqref{eq:Hsequation}, see \cite{Biswas}. Other applications of master equations
can be found in \cite{Allen-Caffarelli-Vasseur,Caffarelli-Silvestre-Master,Stinga-Torrea-SIAM} and references therein. 

As we will show, master equations as in \eqref{eq:Hsequation} are
nonlocal in space and time, and take into account the past (memory). Generally speaking, these equations take the form
\begin{equation}\label{eq:master}
\int_0^\infty\int_{\Omega}(u(t-\tau,z)-u(t,x))K(t,x,\tau,z)\,dz\,d\tau=f(t,x)
\end{equation}
for $t\in\R$ and $x\in\Omega$, where $K$ is some kernel. These are also related to continuous
time random walks, see \cite{MK}.
L.~A.~Caffarelli and L.~Silvestre proved H\"older estimates for equations as in \eqref{eq:master}
when the right hand side $f$ is bounded, see \cite{Caffarelli-Silvestre-Master}.
They assumed some structural conditions on the kernel $K$ that enforce regularity of $u$.
On the other hand, the most basic master equation is given by the fractional powers of the 
heat operator $(\partial_t-\Delta)^su=f$, and this case was analyzed in great detail in \cite{Stinga-Torrea-SIAM}.

In general, the elliptic operator $L$ we consider in \eqref{eq:Hsequation} is a nonnegative operator of the form
\begin{equation}\label{eq:operatorL}
L=-\dvie(a(x)\nabla) + c(x)
\end{equation}
in a domain $\Omega\subseteq{\R^n}$ that may be unbounded.
Here $a(x)=(a^{ij}(x))$ is a bounded, measurable, symmetric matrix defined in $\Omega$, satisfying
the uniform ellipticity condition, that is, for some $\Lambda\geq1$,
$$\Lambda^{-1}|\xi|^2\leq a^{ij}(x)\xi_i\xi_j\leq\Lambda|\xi|^2$$
for a.e. $x\in\Omega$, for all $\xi\in\R^n$, and $c(x) \in L^{\infty}_{\mathrm{loc}}(\Omega)$.
Some concrete operators we consider in \eqref{eq:operatorL} are the following:
	\begin{enumerate}[$(1)$]
		\item $L=-\dvie(a(x)\nabla) + c(x) $ in a bounded domain $\Omega$ with homogeneous Dirichlet or Neumann (conormal)
		boundary condition. The potential function $c(x) \geq 0$ and $c(x) \in L^{\infty}(\Omega)$. If $c(x) = 0$ and $a(x) = I$, then we get $-\Delta_D$ and $-\Delta_N$,
		the Dirichlet and Neumann Laplacians, respectively.  \smallskip
		\item The harmonic oscillators $L=-\Delta+|x|^2$ and $L=-\Delta+|x|^2-n$ in $\Omega=\R^n$. \smallskip
		\item The Laguerre differential operator $L =\frac{1}{4}( - \Delta + |x|^2 + \sum^n_{i=1} \frac{1}{x^2_i} \left( \alpha^2_i - \frac{1}{4} \right))$,
		for $\alpha_i>-1$, in $\Omega=(0, \infty)^n$. \smallskip
		\item The ultraspherical operator $L=- \frac{d^2}{d x^2}+\frac{\lambda(\lambda -1)}{\sin^2 x}$, for $\lambda>0$, in $\Omega=(0, \pi)$.
		\smallskip
		\item The Laplacian $-\Delta$ in $\Omega=\R^n$.  \smallskip
		\item The Bessel operator $L=-\frac{d^2}{d x^2}+\frac{\lambda(\lambda -1)}{x^2}$, for $\lambda>0$, in $\Omega=(0,\infty)$.
	\end{enumerate}
	Observe that in $(2)$--$(6)$ the ellipticity constant is $\Lambda=1$. The operators $(2)$--$(4)$ arise in classical orthogonal expansions and
	the Bessel operator in $(6)$ appears when considering radial-in-space solutions to $(\partial_t-\Delta)^su=f$.

	The precise notion of $H^s=(\partial_t+L)^s$ is a delicate point. Indeed,
		a definition in terms of the Fourier transform in time and the spectral resolution of $L$
		leads to considering the multi-valued complex function $z\mapsto z^s$.
		In this paper, we develop a semigroup approach that allows us to overcome this difficulty.
		Observe that, unlike the case treated in \cite{Stinga-Torrea-SIAM}, where $L=-\Delta$,
		the Fourier transform in space is not the most useful tool anymore because $L$ is not
		translation invariant and has nonsmooth coefficients, and, in general, the domain $\Omega$ is not the whole space $\R^n$.
		In addition, $H^su$ needs to be understood
		in the weak sense, see Section \ref{section:pointwise} for more details.
		
By using our semigroup method for the concrete cases $(1)$--$(6)$ we are able to obtain an integro-differential formula
for $H^su$ which shows that \eqref{eq:Hsequation} is indeed a master equation as in \eqref{eq:master}, but
in divergence form.

	\begin{thm}\label{thm:formula}
		Let $L$ be as in  $(1)$--$(6)$. If $u,v\in \Dom(H^s)\cap C^\infty_c(\R\times\Omega)$ then
		\begin{align*}
		\langle &H^su,v\rangle=\langle(\partial_t+L)^su,v\rangle \\
		&=\int_0^\infty\int_{\R}\int_\Omega\int_{\Omega} K_s(\tau,x,z)(u(t-\tau,x)-u(t-\tau,z))({{v(t,x)-v(t,z)}})\,dz\,dx\,dt\,d\tau \\
		&\quad+\int_0^\infty\bigg[\int_\R\int_\Omega\frac{\big(1-e^{-\tau L}1(x)\big)}{|\Gamma(-s)|\tau^{1+s}}u(t,x){{v(t,x)}}\,dx\,dt\\
		&\quad\qquad\qquad-\int_\R\int_\Omega e^{-\tau L}1(x)\frac{(u(t-\tau,x)-u(t,x))}{|\Gamma(-s)|\tau^{1+s}}{{v(t,x)}}\,dx\,dt\bigg]\,d\tau \blue{,}
		\end{align*}
		where
		$$K_s(\tau,x,z)=\frac{W_\tau(x,z)}{2|\Gamma(-s)|\tau^{1+s}},$$
		$W_\tau(x,z)$ is the heat kernel for $L$, and
		$$e^{-\tau L}1(x)=\int_\Omega W_\tau(x,z)\,dz.$$
	\end{thm}
	
	\begin{rem}
	There are cases in which $e^{-\tau L}1(x)\equiv 1$. This occurs, for example, when $L$ is as in \eqref{eq:operatorL}
	with $c(x)=0$ and has either Neumann boundary condition or $\Omega=\R^n$, or when $L$ is
	the Laplacian $-\Delta$ on $\R^n$. Then, in Theorem \ref{thm:formula} we get
	\begin{align*}
	\langle H^su,v\rangle &=\int_0^\infty\int_{\R}\int_\Omega
	\int_\Omega K_s(\tau,x,z)(u(t-\tau,x)-u(t-\tau,z))({{v(t,x)-v(t,z)}})\,dz\,dx\,dt\,d\tau \\
	&\quad-\int_0^\infty\int_\R\int_\Omega \frac{(u(t-\tau,x)-u(t,x))}{|\Gamma(-s)|\tau^{1+s}}{{v(t,x)}}\,dx\,dt\,d\tau.
	\end{align*}
	The second integral term above is equal to
	$$-\int_\R\int_\Omega (D_{\mathrm{left}})^su(t,x){{v(t,x)}}\,dx\,dt$$
	where $(D_{\mathrm{left}})^s$ denotes the fractional power of the derivative from the left,
	which coincides with the Marchaud fractional derivative, acting on the variable $t\in\R$, see \cite{Bernardis}.
\end{rem}

\begin{rem}\label{rem:gaussian}
If the heat kernel $W_\tau(x,z)$ of $L$ satisfies Gaussian bounds,
then we can obtain pointwise estimates for the kernel $K_s(\tau,x,z)$ in Theorem \ref{thm:formula}.
Here we present two cases for an operator $L$ as in $(1)$ for $c\equiv0$.
\begin{enumerate}[(a)]
\item If $\Omega$ is a bounded domain and the coefficients $a(x)$ are bounded and measurable then, by
the results of \cite{Davies}, we have the Gaussian upper bound
$$W_\tau(x,z)\leq\frac{C}{\tau^{n/2}}e^{-|x-z|^2/(c\tau)}\qquad x,z\in\Omega,~\tau>0,$$
for some constants $C,c>0$. From here, it readily follows the estimate
$$K_s(\tau,x,z)\le \frac{\Lambda}{|x-z|^{n+2+2s}+\tau^{n/2+1+s}},\qquad\hbox{for every}~x,z\in\Omega,~\tau>0,$$
for some $\Lambda>0$.
\item If $\Omega=\R^n$ and the coefficients $a(x)$ are bounded and measurable
then, by Aronson's estimates \cite{Aronson}, we have the two-sided heat kernel Gaussian bound
$$\frac{C_1}{\tau^{n/2}}e^{-|x-z|^2/(c_1\tau)} \leq W_\tau(x,z)\leq\frac{C_2}{\tau^{n/2}}e^{-|x-z|^2/(c_2\tau)}\qquad x,z\in\R^n,~\tau>0,$$
for some constants $C_1,c_1,C_2,c_2>0$. Then, an upper bound for $K_s(\tau,x,z)$ as in $\mathrm{(a)}$ holds,
and we also have the lower bound
$$K_s(\tau,x,z)\ge \frac{\lambda}{|x-z|^{n+2+2s}},\qquad\hbox{when}~\tau\sim |x-z|^2,$$
for some $\lambda>0$.
\end{enumerate}
These estimates show that $(\partial_t+L)^s$ is an equation of order $s$ in time and $2s$ in space.
\end{rem}

For master equations \eqref{eq:Hsequation} with a general $L$ as in \eqref{eq:operatorL}
we prove parabolic interior and boundary Harnack inequalities, local boundedness and
parabolic H\"older regularity. For notation see Section \ref{section:pointwise}.

\begin{thm}[Parabolic interior Harnack inequality]\label{harn_para}
	Let $L$ be as in \eqref{eq:operatorL}.
	Let $B_{2r}$ be a ball of radius $2r$, $r>0$, such that $B_{2r}\subset\subset\Omega$.
	There exists a constant $c>0$ depending only on $n$, $s$, $\Lambda$ and $r$
	such that if $u= u(t,x)\in\Dom(H^s)$ is a solution to
	$$\begin{cases}
	H^su= 0&\hbox{for}~(t,x)\in R:=(0,1)\times B_{2r}\\
	u\geq 0&\hbox{for}~(t,x)\in(-\infty,1)\times\Omega,
	\end{cases}$$
	then 
	$$\sup_{R^{-}} u \leq c \inf_{R^{+}} u$$
	where $R^-:= (1/4,1/2)\times B_{r}$ and $R^+:=(3/4,1)\times B_r$.
	Moreover, solutions $u\in\Dom(H^s)$ to $H^su=0$ in $R$ are locally bounded
	and locally parabolically $\alpha$-H\"older continuous in $R$, for some exponent
	$0<\alpha<1$ depending on $n$, $\Lambda$ and $s$. More precisely, for any
	compact set $K\subset R$ there exists $C=C(c,K,R)>0$ such that 
	$$\|u\|_{C^{\alpha/2,\alpha}_{t,x}(K)}\leq C\|u\|_{L^2(\R\times\Omega)}.$$
\end{thm}
To present the parabolic boundary Harnack inequality, let $\Omega_0\subset\Omega$ and
$\tilde{x}\in\partial\Omega_0$ such that $B_{{2r}}(\tilde{x})\subset\Omega$, for some $r>0$ fixed.
Suppose that, up to a rotation and translation, ${B_{2r}(\tilde{x})} \cap \partial \Omega_0$
can be represented as the graph of a Lipschitz function $g:\R^{n-1}\to\R$ in the $e_n=(0,\ldots,0,1)$-direction,
such that $g$ has Lipschitz constant $M>0$. Thus,
$$\Omega_0\cap {B_{2r}(\tilde{x}) } = \{(x',x_n):\: x_n>g(x')\}\cap {B_{2r}(\tilde{x}) }$$
$$\partial \Omega_0 \cap B_{2r}(\tilde{x})  = \{ (x',x_n) :\,\, x_n = g(x') \} \cap  B_{2r}(\tilde{x}) .$$
Fix a point $(t_0,x_0) \in (-2,2) \times \Omega_0$ such that $t_0>1$.

\begin{thm}[Parabolic boundary Harnack inequality]\label{harn_bd}
	Let $L$  be as in \eqref{eq:operatorL}.
	Assume the geometric conditions on $\Omega_0$ and $\Omega$ described above. Then there exists a constant $C>0$
	depending on $n$, $\Lambda$, $r$, $M$, $s$, $t_0-1$ and $g$, such that if
	$u(t,x)\in\Dom(H^s)$ is a solution to
	$$\begin{cases}
	H^s u = 0&\hbox{for}~(t,x)\in(-2,2)\times(\Omega_0\cap {B_{2r}(\tilde{x}) })\\
	u \geq 0&\hbox{for}~(t,x)\in(-\infty,2)\times\Omega
	\end{cases}$$
	such that $u$ vanishes continuously on $(-2,2) \times ((\Omega\setminus \Omega_0) \cap {B_{2r}(\tilde{x}) })$ then
	$$\sup_{(-1,1) \times (\Omega_0\cap {B_r(\tilde{x}) })} u(t,x) \leq C u(t_0, x_0).$$
\end{thm}

The main tool to prove Theorems \ref{harn_para} and \ref{harn_bd} is an extension problem characterization
for the fractional operators $(\partial_t+L)^s$. 
Observe that, in general, $L$ as in \eqref{eq:operatorL} may have discrete or continuous spectrum in different
Hilbert spaces. The extension problem we present here  not only works for  \eqref{eq:operatorL},
but for any fractional operator of the form $(\partial_t+L)^s$, where $L$ is a 
nonnegative normal linear operator in a Hilbert space $L^2(\Omega)$ with some positive measure $d\eta$.

\begin{thm}[Extension problem]\label{th_exten}
	Let $L$ be a nonnegative normal linear operator on $L^2(\Omega)$
	and $H = \partial_t + L$. Let $u\in\Dom(H^s)$. For $(t,x)\in\R\times\Omega$ and $y>0$ we define
	\begin{equation}\label{solution}
	\begin{aligned}
	U(t,x,y) &= \frac{y^{2s}}{4^s \Gamma(s)} \int^{\infty}_0 e^{-y^2/(4t)} e^{-\tau H} u(t,x)\,\frac{d \tau}{\tau^{1+s}} \\
	&= \frac{1}{\Gamma(s)}\int_0^\infty e^{-r}e^{-\frac{y^2}{4r}H}u(t,x)\,\frac{dr}{r^{1-s}}\\
	&= \frac{1}{\Gamma(s)}\int_0^\infty e^{-y^2/(4r)}e^{-rH}(H^su)(t,x)\,\frac{dr}{r^{1-s}}.
	\end{aligned}
	\end{equation}
	Then $U(\cdot,\cdot,y) \in \Dom(H)$ for each $y>0$,
	$U\in C^\infty((0,\infty);L^2(\R \times \Omega)) \cap C([0, \infty);L^2( \R \times \Omega))$
	and $U\in L^2((0,\infty);\Dom(H),y^{1-2s}dy)$
	Moreover, $U$ is a solution to
	\begin{equation}\label{eq:localU}
	\begin{cases}
	\langle HU,v\rangle=\Big\langle\frac{1-2s}{y}\partial_yU+\partial_{yy}U,v\Big\rangle_{L^2(\R\times\Omega)}&\hbox{for each}~
	v\in\Dom(H)~\hbox{and}~y>0\\
	\displaystyle\lim_{y\to0^+}U(t,x,y)=u(t,x)&\hbox{in}~L^2(\R\times\Omega)
	\end{cases}
	\end{equation}
	such that
	$$\lim_{y \to \infty} \langle U, v \rangle_{L^2(\R \times \Omega)} = 0,\quad\hbox{for every}~v\in L^2(\R\times\Omega)$$
	and
	$$ \sup_{y>0}|\langle y^{1-2s}\partial_y U, v \rangle_{L^2(\R \times \Omega)} |\leq C_s\|u\|_{H^s}\|v\|_{H^s},\quad\hbox{for every}~v\in\Dom(H^s).$$
	In addition, for every $v \in \Dom(H^s)$, 
	\begin{align*}
	-\frac{1}{2s}\lim_{y\to 0^+}\langle y^{1-2s}\partial_y U,v\rangle_{L^2(\R\times\Omega)}
	&=\frac{|\Gamma(-s)|}{4^s\Gamma(s)}\langle H^{s}u,v\rangle \\
	&= -\lim_{y\to0^+}\Big\langle\frac{U(\cdot,\cdot,y)-U(\cdot,\cdot,0)}{y^{2s}},v\Big\rangle_{L^2(\R\times\Omega)}.
	\end{align*}
\end{thm}

Theorem \ref{th_exten} shows that the solution $u$ to the nonlocal problem $H^su=f$
is characterized by the solution $U$ to the local problem \eqref{eq:localU}. Another main
novelty is the set of explicit formulas for the solution $U$ we discovered in \eqref{solution}.
Observe that $U$ is given in terms of the semigroup
generated by $H$ acting either on $u$ or on $f=H^su$.

Extension problems as the one in Theorem \ref{th_exten} have proven to be very useful for several applications.
For example, the extension problem allows to find a monotonicity formula and prove regularity estimates for
free boundary problems for fractional powers of the heat operator \cite{ACM}.
Also, they are central tools for the numerical analysis of fractional equations using finite
elements methods \cite{NOS}.

The reader will notice that our semigroup methodology is very general and has wide applicability. Indeed,
we can also consider other master equations $(\partial_t+L)^su=f$
in different settings. For example, the elliptic operator $L$ can be the Laplace--Beltrami or the conformal Laplacian
on a manifold, a subelliptic operator on a Lie group (like a Carnot or Heisenberg group), the Laplacian
in infinite dimensions (Wiener space), the Laplacian on a lattice, or nonsymmetric operators like
nondivergence form elliptic operators and operators associated to Dirichlet forms in Hilbert spaces.
More generally, it is enough for $L$ to be the generator of a uniformly bounded $C_0$-semigroup as in \cite{Gale-Miana-Stinga}
because, in that case, the semigroup $\{e^{-\tau H}\}_{\tau\geq0}$ will be well defined.
Applications and generalizations to these cases will appear elsewhere.

In view of our results, we expect parabolic Harnack inequalities to hold
for \eqref{eq:master} in the general case of \cite{Caffarelli-Silvestre-Master}.
This remains an open problem for which a different technique
that is not based on extension problems needs to be found.

In Section \ref{trans_sec} we develop a \textit{transference method} for fractional powers of parabolic operators
that allows us to transfer the Harnack inequalities and H\"older estimates for $(\partial_t+L)^su=f$ from Theorems \ref{harn_para}
and \ref{harn_bd} to other master equations of the form $(\partial_t+\bar{L})^s\bar{u}=\bar{f}$.
Here, formally, $\bar{L}= (U\circ W)^{-1}\circ L\circ(U\circ W)$, where $U$ is a multiplication operator by a smooth positive function and $W$ is a smooth change of variables operator. This method is particularly useful when $\bar{L}$ is one of the following
elliptic operators having gradient term.
\begin{enumerate}
\item[$(7)$] The Ornstein--Uhlenbeck operator $L=-\Delta+2x\cdot\nabla$ in $\Omega=\R^n$ with the Gaussian measure.
\item[$(8)$] The Laguerre operators
\begin{itemize}
\item $L=\sum^{n}_{i=1} \left( -x_i \frac{\partial^2}{\partial x_i^2} - ( \alpha_i + 1 )\frac{\partial}{\partial x_i}
+\frac{x_i}{4} \right)$,
\item $L=\frac{1}{4}(-\Delta+|x|^2-\sum_{i=1}^n\frac{2\alpha_i+1}{x_i}\frac{\partial}{\partial_{x_i}})$,
\item $L=\sum^{n}_{i=1} \left( -x_i \frac{\partial^2}{\partial x_i^2} - \frac{\partial}{\partial x_i}+\frac{x_i}{4}
+\frac{\alpha_i^2}{4x_i} \right)$,
\item $L = \sum^{n}_{i=1} \left( -x_i \frac{\partial^2}{\partial x_i^2} - ( \alpha_i + 1 - x_i)\frac{\partial}{\partial x_i} \right)$,
\end{itemize}
for $\alpha_i>-1$ in $\Omega=(0,\infty)^n$, with their corresponding Laguerre measures.
\item[$(9)$] The ultraspherical operator $L = - \frac{d^2}{d x^2} -2 \lambda \cot x\frac{d}{dx} + \lambda^2$, for $\lambda>0$ in $\Omega=(0,\pi)$
with the measure $d\eta(x)=\sin^{2\lambda}x\,dx$.
\item[$(10)$] The Bessel operator $L = -\frac{d^2}{d x^2} -\frac{2 \lambda}{x} \frac{d}{dx}$ , for $\lambda>0$ in $\Omega=(0,\infty)$
with the measure $d\eta(x)=x^{2\lambda}dx$.
\end{enumerate}
These are related to classical orthogonal expansions and can be obtained by transference from
the operators $L$ listed in $(2)$--$(6)$. Transference techniques in the elliptic case have been widely used in harmonic analysis,
see \cite{AMST}, and also for fractional elliptic PDEs, see \cite{SZ}.
We also point out that pointwise formulas for the nonlocal operators $(\partial_t+L)^su(t,x)$
when $L$ is as in $(7)$--$(10)$
can be deduced exactly as in Theorem \ref{thm:formula} by using the corresponding heat kernels.
Details are left to the interested reader.

\begin{thm}[Transference method]\label{thm:transference}
	If Theorems \ref{harn_para} and \ref{harn_bd} hold true for solutions $u\in\Dom(H^s)$ to $(\partial_t+L)^su=0$, where $L$ is as in \eqref{eq:operatorL},
	then they also hold true for solutions $\bar{u}\in\Dom(\bar{H}^s)$ to $(\partial_t+\bar{L})^s\bar{u}=0$.
\end{thm}

A very detailed study of the master equation $(\partial_t-\Delta)^su=f$ in $\Omega=\R^n$
was carried out in \cite{Stinga-Torrea-SIAM}. There are many significant challenges in our work with respect to that one.
As we mentioned before, one notices that while defining the operator $(\partial_t - \Delta)^s$ in \cite{Stinga-Torrea-SIAM}, the authors use the
Fourier transform both the space and times variables. They are able to find explicit pointwise formulas
for $(\partial_t -\Delta)^s u(t,x)$ and also for the inverse $(\partial_t-\Delta)^{-s}f(t,x)$.
In general, our operators $L$ have nonsmooth coefficients and the domains $\Omega$ are not $\R^n$. Therefore,
the use of the Fourier transform in the space variable is not the most adequate approach.
Instead, we take advantage of the spectral resolution of $L$. As a result of our generality, we do not have any explicit expression for the kernel
in the master equation formulation, see Theorem \ref{thm:formula}, rather bounds coming from heat kernel estimates,
see Remark \ref{rem:gaussian}.
Theorem \ref{th_exten} for $(\partial_t-\Delta)^s$ was also proved in \cite{Stinga-Torrea-SIAM},
and an explicit formula for $U(t,x,y)$ was given using a Poisson kernel. In such expression for $U$, one can directly 
take derivatives inside the integral sign to prove that $U$ is a classical solution and satisfies the extension equation pointwise.
In turn, in our case, to prove that $U$ is a solution we need to establish the weak formulation, see Sections \ref{section3} and \ref{div_sec}.
In terms of regularity, \cite{Stinga-Torrea-SIAM} uses
the symmetries of the Laplacian and the heat semigroup characterizations of H\"older and Zygmund spaces
to prove parabolic Schauder estimates in a very elegant, quick way.
In contrast, in \cite{B-S}, a compactness method needed to be developed to prove interior and boundary Schauder estimates for
\eqref{eq:Hsequation}. Finally, in this paper 
we are able to transfer Harnack inequalities from one set of operators to another set of operators
because our Theorems \ref{harn_para} and \ref{harn_bd} hold for very general elliptic operators $L$.
This would not be possible by just using the results in \cite{Stinga-Torrea-SIAM}.

The paper is organized as follows.
In Section \ref{section:pointwise} we provide the precise definition of $(\partial_t+L)^s$ and
prove Theorem \ref{thm:formula}. 
Section \ref{section3} contains the proof of the general parabolic extension problem (Theorem \ref{th_exten})
and Section \ref{div_sec} explains how to apply it when $L$ is an elliptic operator in divergence form.
The proof of Theorems \ref{harn_para} and \ref{harn_bd} are given in Section \ref{proof_harn_para}.
Finally, the transference method (Theorem \ref{thm:transference}) and the proof of Harnack inequalities and
H\"older estimates for $(\partial_t+{L})^su=f$, where $L$ is as in $(7)$--$(10)$, are presented in Section \ref{trans_sec}.

\section{Definition and integro-differential formula}\label{section:pointwise}

In this section we present the precise definition of $H^su(t,x)=(\partial_t+L)^su(t,x)$ and show
that in general this is a master operator.

Let $L$ be a nonnegative normal linear operator on a Hilbert space $L^2(\Omega)$
with some positive measure $d\eta$.
For concreteness and simplicity of the presentation, we will assume that $L$ has discrete spectrum
and $d\eta$ is the Lebesgue measure. We will also assume for simplicity that the eigenfunctions of $L$ are real-valued.
We can always obtain the general result by using the Spectral Theorem,
the Fourier transform, the Hankel transform, the corresponding orthogonal expansions with respect
to $d\eta$, etc.

Therefore, suppose that $L$ has a countable sequence of eigenvalues and eigenfunctions $(\lambda_k,\phi_k)_{k\geq0}$
such that $0\leq\lambda_0<\lambda_1\leq \lambda_2\leq \cdots\nearrow\infty$ and so that $\{\phi_k\}_{k\geq0}$
forms an orthonormal basis of $L^2(\Omega)$.   In the case in which $\lambda_0=0$ (for instance, for the Neumann Laplacian) we assume that all the functions involved
have zero integral mean over $\Omega$.
 With this, any function $u(t,x)\in L^2(\R\times\Omega)$ can be written as
$$u(t,x)=\frac{1}{(2\pi)^{1/2}}\int_{\R}\sum_{k=0}^\infty\widehat{u_k}(\rho)\phi_k(x)e^{it\rho}\,d\rho,$$
where
$$u_k(t)=\int_{\Omega}u(t,x)\phi_k(x)\,dx$$
and $\widehat{u_k}(\rho)$ is the Fourier transform of $u_k(t)$ with respect to the variable $t\in\R$:
$$\widehat{u_k}(\rho)=\frac{1}{(2\pi)^{1/2}}\int_\R u_k(t)e^{-i\rho t}\,dt.$$
The domain of the operator $H^s\equiv(\partial_t+L)^s$, $0\leq s\leq1$, is defined as
$$\Dom(H^s)=\bigg\{u\in L^2(\R\times\Omega):
\|u\|_{H^s}^2:=\int_{\R}\sum_{k=0}^\infty|i\rho+\lambda_k|^s|\widehat{u_k}(\rho)|^2\,d\rho<\infty\bigg\}.$$

For $u\in\Dom(H^s)$ we define
$H^su$ as a bounded linear functional on $\Dom(H^s)$ that acts on any $v\in\Dom(H^s)$ by
\begin{equation}\label{eq:definitionHs}
\langle H^su,v\rangle\equiv\int_{\R}\sum_{k=0}^\infty(i\rho+\lambda_k)^s\widehat{u_k}(\rho)\overline{\widehat{v_k}(\rho)}\,d\rho
\end{equation}
where $\overline{\widehat{v_k}(\rho)}$ denotes the complex conjugate of $\widehat{v_k}(\rho)$. We have
$$\|u\|_{H^s}^2=\langle H^{s/2}u,H^{s/2}u\rangle\quad\hbox{for any}~0\leq s\leq1.$$
Notice that we need to appropriately decide which $s$-power of the complex number
$(i\rho+\lambda_k)$ we are taking. We are able to clarify this by developing a semigroup technique,
in which the Gamma function plays a crucial role. The method permits us to show that \eqref{eq:definitionHs}
is indeed a master equation, or nonlocal in space and time integro-differential operator, in divergence form.
Observe as well that $\Dom(H^s)$ encodes the boundary condition on $L$.

As the family of eigenfunctions $\{\phi_k\}_{k\geq0}$ is an orthonormal basis of $L^2(\Omega)$, we can write
the semigroup $\{e^{-\tau L}\}_{\tau\geq0}$ generated by $L$ as
$$\langle e^{-\tau L}\varphi,\psi\rangle_{L^2(\Omega)} =\sum_{k=0}^\infty e^{-\tau \lambda_k}\varphi_k\psi_k
=\int_\Omega\int_\Omega W_\tau (x,z)\varphi(z){{ \psi(x) }}\,dz\, dx$$ for any $\varphi,\psi\in L^2(\Omega)$, where $\displaystyle\varphi_k=\int_{\Omega}\varphi\phi_k\,dx$ and $\displaystyle\psi_k=\int_{\Omega}\psi\phi_k\,dx$.
As it happens for \eqref{eq:operatorL} and all the other cases $(1)$--$(9)$, we will always assume that the heat kernel for $L$ is
symmetric and nonnegative:
 $$W_\tau(x,z)=W_\tau(z,x)\geq0.$$

Since $\partial_t$ and $L$ commute, we define, for any $u\in L^2(\R\times\Omega)$,
$$e^{-\tau H}u(t,x)=e^{-\tau L}(e^{-\tau \partial_t} u)(t,x)=e^{-\tau L}(u(t-\tau,\cdot))(x)$$
in the sense that, for any $v\in L^2(\R\times\Omega)$, 
\begin{equation}\label{eq:semigroupL2}
\begin{aligned}
\langle e^{-\tau H}u,v\rangle_{L^2(\R\times\Omega)} &= \int_\R\sum_{k=0}^\infty e^{-\tau(i\rho+\lambda_k)}\widehat{u_k}(\rho)\overline{\widehat{v_k}(\rho)}\,d\rho \\
&=\int_\R\sum_{k=0}^\infty e^{-\tau\lambda_k}u_k(t-\tau)v_k(t)\,dt\\
&=\int_\R\int_\Omega \int_\Omega W_\tau (x,z) u(t-\tau,z){{v(t,x)}} \,dz\, dx\,dt.
\end{aligned}
\end{equation}

\begin{lem}\label{lem:semigroup}
		Let $0<s<1$. If $u\in\Dom(H^s)$ then 
	$$H^su =  \frac{1}{\Gamma(-s)} \int^{\infty}_0 \left( {e^{-\tau H} u }- {u} \right) \frac{d \tau} {\tau^{1+s}}$$
		in the sense that, for any $v\in\Dom(H^s)$,
$$ \langle H^s u , v \rangle =  \frac{1}{\Gamma(-s)} \int^{\infty}_0 \left(\langle {e^{-\tau H}u} , v \rangle_{L^2(\R\times\Omega)} - \langle {u} , {v} \rangle_{L^2{(\R\times\Omega)}} \right) \frac{d \tau}{\tau^{1+s}}.$$
	\end{lem}
	\begin{proof}
		Let $u,v \in \Dom(H^s)$. We will use the following numerical formula with the Gamma function
		that comes from performing the analytic continuation to $\Re(z)>0$ of the function
		that maps $t\in[0,\infty)$ to $t^s$, see \cite{Bernardis, Stinga-Torrea-SIAM},
\begin{equation}\label{gamma_com}
(i\rho + \lambda_k)^s = \frac{1}{\Gamma(-s)} \int^{\infty}_0 (e^{-\tau(i\rho+ \lambda_k)}- 1) \frac{d \tau}{\tau^{1+s}},\quad\rho\in\R.
\end{equation}
The integral above is absolutely convergent. Then, in \eqref{eq:definitionHs} we have
		\begin{align*}
	\langle H^s u, v \rangle&=	\int_{\mathbb{R}}\sum_{k=0}^\infty\left[ \frac{1}{\Gamma(-s)} \int^{\infty}_0 (e^{-\tau(i\rho+ \lambda_k)}- 1)\frac{d \tau}{\tau^{1+s}} \right]
	\widehat{u_k}(\rho)\overline{\widehat{v_k}(\rho)}\, d \rho. 
		\end{align*}
		On one hand,
		$$\int_0^{1/|i\rho+\lambda_k|}|e^{-\tau(i\rho+\lambda_k)}-1|\,\frac{d\tau}{\tau^{1+s}}\leq C|i\rho+\lambda_k|\int_0^{1/|i\rho+\lambda_k|}\tau^{-s}\,d\tau
		=C|i\rho+\lambda_k|^s.$$
		On the other hand,
		$$\int_{1/|i\rho+\lambda_k|}^\infty|e^{-\tau(i\rho+\lambda_k)}-1|\,\frac{d\tau}{\tau^{1+s}}\leq
		C\int_{1/|i\rho+\lambda_k|}^\infty\tau^{-1-s}\,d\tau
		=C|i\rho+\lambda_k|^s.$$
		Since $u,v\in\Dom(H^s)$, Fubini's Theorem and \eqref{eq:semigroupL2} allow us to get the conclusion.
\end{proof}

\begin{proof}[Proof of Theorem \ref{thm:formula}]
For $u,v\in \Dom(H^s)\cap C^\infty_c(\R\times\Omega)$ we have, by Lemma \ref{lem:semigroup}, up to the multiplicative constant $1/\Gamma(-s)$,
\begin{align*}
\langle H^s&u,v\rangle =\int_0^\infty\big(\langle e^{-\tau L}u(\cdot-\tau,\cdot),v(\cdot,\cdot)\rangle_{L^2(\R\times\Omega)}-\langle u,v\rangle_{L^2(\R\times\Omega)}\big)
\,\frac{d\tau}{\tau^{1+s}}\\
&=\int_0^\infty\bigg[\int_\R\int_\Omega\int_\Omega W_\tau(x,z)u(t-\tau,z){{v(t,x)}}\,dz\,dx\,dt-\int_{\R}\int_\Omega u(t,x){{v(t,x)}}\,dx\,dt\bigg]\,\frac{d\tau}{\tau^{1+s}}.
\end{align*}
The integral in brackets can be rewritten as
\begin{equation}\label{eq:1}
\begin{aligned}
\int_\R&\int_\Omega\int_\Omega W_\tau(x,z)(u(t-\tau,z)-u(t-\tau,x)){{v(t,x)}}\,dz\,dx\,dt\\
&+\int_\R\int_\Omega\big(e^{-\tau L}1(x)u(t-\tau,x)-u(t,x)\big){{v(t,x)}}\,dx\,dt.
\end{aligned}
\end{equation}
By exchanging the roles of $x$ and $z$ and using that $W_\tau(z,x)=W_\tau(x,z)$, the integrals above are also equal to
\begin{equation}\label{eq:2}
\begin{aligned}
-\int_\R&\int_\Omega\int_\Omega W_\tau(x,z)(u(t-\tau,z)-u(t-\tau,x)){{v(t,z)}}\,dx\,dz\,dt\\
&+\int_\R\int_\Omega\big(e^{-\tau L}1(x)u(t-\tau,x)-u(t,x)\big){{v(t,x)}}\,dx\,dt.
\end{aligned}
\end{equation}
By adding \eqref{eq:1} and \eqref{eq:2}, we get that, up to the multiplicative constant $1/|\Gamma(-s)|$,
\begin{align*}
2\langle H^su,v\rangle &=\int_0^\infty\bigg[\int_\R\int_\Omega\int_\Omega W_\tau(x,z)(u(t-\tau,x)-u(t-\tau,z))({{v(t,x)-v(t,z)}})\,dz\,dx\,dt\\
&\qquad\qquad+2\int_\R\int_\Omega\big(u(t,x)-e^{-\tau L}1(x)u(t-\tau,x)\big){{v(t,x)}}\,dx\,dt\bigg]\,\frac{d\tau}{\tau^{1+s}}.
\end{align*}
For the operators $L$ in $(1)$--$(6)$
we always have the Gaussian estimate
$$|W_\tau(x,z)|\le C\frac{e^{-|x-z|^2/(c\tau)}}{\tau^{n/2}}$$
(see, for instance, \cite{AMST,Aronson,Stinga-Caffa,Davies}).
Observe that $u,v$ can be extended by zero outside of $\R\times\Omega$ so
we can regard them as functions in $C^\infty_c(\R^{n+1})$. Then
\begin{align*}
\bigg|&\int_0^\infty\int_\Omega\int_\Omega
W_\tau(x,z)\int_\R(u(t-\tau,x)-u(t-\tau,z))({{v(t,x)-v(t,z)}})\,dt\,dz\,dx\,\frac{d\tau}{\tau^{1+s}}\bigg|\\
&=\bigg|\int_0^\infty\int_\Omega\int_\Omega W_\tau(x,z)\int_\R e^{i\tau\rho}(\widehat{u}(\rho,x)-\widehat{u}(\rho,z))\overline{(\widehat{v}(\rho,x)-\widehat{v}(\rho,z))}
\,d\rho\,dz\,dx\,\frac{d\tau}{\tau^{1+s}}\bigg|\\
&\leq \int_\R\int_{\Omega}\int_\Omega|\widehat{u}(\rho,x)-\widehat{u}(\rho,z)||\widehat{v}(\rho,x)-\widehat{v}(\rho,z)|\bigg[\int_0^\infty W_\tau(x,z)\,\frac{d\tau}{\tau^{1+s}}\bigg]\,dz\,dx
\,d\rho \\
&\leq C\int_\R\int_{\R^n}\int_{\R^n}\frac{|\widehat{u}(\rho,x)-\widehat{u}(\rho,z)|^2}{|x-z|^{n+2s}}\,dz\,dx\,d\rho
+C\int_\R\int_{\R^n}\int_{\R^n}\frac{|\widehat{v}(\rho,x)-\widehat{v}(\rho,z)|^2}{|x-z|^{n+2s}}\,dz\,dx\,d\rho \\
&=C\int_{\R}\big(\|(-\Delta)^{s/2}\widehat{u}(\rho,\cdot)\|_{L^2(\R^n)}^2+\|(-\Delta)^{s/2}\widehat{v}(\rho,\cdot)\|_{L^2(\R^n)}^2\big)\,d\rho \\
&=C\int_{\R^{n+1}}|\xi|^{2s}\big(|\mathcal{F}_{\R^{n+1}}(u)(\rho,\xi)|^2+|\mathcal{F}_{\R^{n+1}}(v)(\rho,\xi)|^2\big)\,d\xi\,d\rho
\end{align*}
where in the last identity we use Plancherel's identity in $\R^n$ and $\mathcal{F}_{\R^{n+1}}$ denotes the Fourier
transform in $(t,x)\in\R^{n+1}$. The last integral above is finite because $u,v\in C^\infty_c(\R^{n+1})$.
Therefore, we can write $\langle H^su,v\rangle$ as
the sum of
$$\frac{1}{2|\Gamma(-s)|}\int_0^\infty\int_\R\int_\Omega\int_\Omega
W_\tau(x,z)(u(t-\tau,x)-u(t-\tau,z))({{v(t,x)-v(t,z)}})\,dz\,dx\,dt\,\frac{d\tau}{\tau^{1+s}}$$
and
$$\frac{1}{|\Gamma(-s)|}\int_0^\infty\int_\R\int_\Omega\big(u(t,x)-e^{-\tau L}1(x)u(t-\tau,x)\big){{v(t,x)}}\,dx\,dt\,\frac{d\tau}{\tau^{1+s}}.$$
The conclusion readily follows from here.
\end{proof}

\begin{rem}
In Theorem \ref{thm:formula} we have assumed that $u$ and $v$ are smooth with compact support. We can relax
this assumption as soon as we are able to show that for any $u,v\in\Dom(H^s)$ we have
$$\int_\R\int_{\Omega}\int_\Omega|\widehat{u}(\rho,x)-\widehat{u}(\rho,z)||\widehat{v}(\rho,x)-\widehat{v}(\rho,z)|\bigg[\int_0^\infty W_\tau(x,z)\,\frac{d\tau}{\tau^{1+s}}\bigg]\,dz\,dx
\,d\rho<\infty.$$
This is true, for instance, in the case when $L$ is as in $(1)$ with either Dirichlet or Neumann
boundary conditions, and with $c(x)=0$. Indeed,
by the results in \cite{Stinga-Caffa}, if $u,v\in\Dom(H^s)$ then it follows that $u,v\in L^2(\R;\Dom(L^s))$.
\end{rem}

\section{Proof of Theorem \ref{th_exten}}\label{section3}

We begin with an important preliminary result.

\begin{lem}\label{MainLem}
Let $0<s<1$. Denote by $K_\nu(z)$ the modified Bessel function of the second kind
and order $\nu$. For $y>0$ and $\lambda\in\C$ with $\Re(\lambda)>0$ we define
\begin{equation}\label{I}
\begin{aligned}
I_s(y,\lambda) &= \frac{2^{1-s}}{\Gamma(s)}(y\sqrt{\lambda})^s K_s(y\sqrt{\lambda}) \\
&=\frac1{\Gamma(s)}\int_0^\infty e^{-t}e^{-\frac{y^2}{4t}\lambda}\,\frac{dt}{t^{1-s}} \\
&= \frac{y^{2s}}{4^s\Gamma(s)}\int_0^\infty e^{-y^2/(4r)}e^{-r\lambda}\,\frac{dr}{r^{1+s}} \\
&= \frac{1}{\Gamma(s)}\int_0^\infty e^{-y^2/(4\tau)}e^{-\tau \lambda}\lambda^s\,\frac{d\tau}{\tau^{1-s}}.
\end{aligned}
\end{equation}
Then the integrals are absolutely convergent. Fix any $s$ and $\lambda$ as above. Then
\begin{enumerate}[$(1)$]
\item $I_s(y,\lambda)$ is a smooth function of $y\in(0,\infty)$.
\item For each $y>0$, $I_s(y,\lambda)$ satisfies the equation
\begin{equation}\label{I_s:eq}
\lambda u-\frac{1-2s}{y}\partial_yu-\partial_{yy}u=0.
\end{equation}
\item $\displaystyle\lim_{y\to0^+}I_s(y,\lambda)=1$.
\item $\displaystyle -y^{1-2s}\partial_yI_s(y,\lambda)=
\frac{\Gamma(1-s)}{4^{s-1/2}\Gamma(s)}\lambda^sI_{1-s}(y,\lambda)$.
\item The following estimates hold:
\begin{enumerate}[$(5.a)$]
\item $\displaystyle|I_s(y,\lambda)|\leq 1$.
\item There is a constant $C_s>0$ such that
$$|I_s(y,\lambda)|\leq C_s(y|\lambda|^{1/2})^{s-1/2}e^{-\cos(\arg(\lambda)/2)y|\lambda|^{1/2}}\quad\hbox{as}~y \to\infty.$$
\item There is a constant $C_s>0$ such that
$$|\lambda I_s(y,\lambda)|+\big|\tfrac{1}{y}\partial_y I_s(y,\lambda)\big|+|\partial_{yy}I_s(y,\lambda)|
\leq C_s\frac{|\lambda|^s}{y^{2-2s}}\quad\hbox{for every}~y>0.$$
\end{enumerate}
\item The function $I_s(\lambda,y)$ is the unique $C^{\infty}$ solution to \eqref{I_s:eq} such that
$$\lim_{y\to0}I_s(y,\lambda)=1,\quad\lim_{y\to\infty}I_s(y,\lambda)=0,\quad\hbox{and}
\quad y^{1-2s}\partial_yI_s(y,\lambda)\in L_y^\infty([0,\infty)).$$
\end{enumerate}
\end{lem}

\begin{proof}
It is well known that for $\nu$ arbitrary (see \cite[eq.~(5.10.25)]{Lebedev})
$$K_\nu(z) = \frac1{2} \Big(\frac{z}{2}\Big)^\nu \int_0^\infty e^{-t} e^{-z^2/4t} t^{-\nu-1} dt \quad \hbox{for}~|\arg z| < \frac{\pi}{4}.$$
As $K_\nu = K_{-\nu}$ we get the second identity in \eqref{I}. Moreover, since $\Re(\lambda)>0$, we have that 
$|e^{-\frac{y^2}{4t}\lambda}|\le 1$, so that the first integral in \eqref{I} is absolutely convergent.
The third identity follows from the change of variables
$r=y^2/(4t)$. The last one for $\lambda>0$ is obtained from the third one via the change of variables $\tau=y^2/(4r\lambda)$,
and the general case of $\mathrm{Re}(\lambda)>0$ follows from the case of $\lambda>0$ by analytic continuation.

Now $(1)$ is easy to check by differentiating under the integral sign. Indeed, since
\begin{equation}\label{eq:differentiation}
|\partial_y(y^{2s}e^{-y^2/(4 \tau)})|=\Big|\Big(2sy^{2s-1} - \frac{y^{2s+1}}{2 \tau}\Big)e^{-y^2/(4 \tau)}\Big|\leq
C_sy^{2s-1}e^{-y^2/(c\tau)},
\end{equation}
we get
$$\partial_yI_s(y,\lambda)= \int_0^\infty \partial_y\bigg(\frac{y^{2s}}{4^s\Gamma(s)}e^{-y^2/(4r)}\bigg)e^{-r\lambda}\,\frac{dr}{r^{1+s}}.$$
Similarly for higher order derivatives. For $(2)$ we can use integration by parts to get
\begin{align*}
\lambda I_s(y,\lambda) &= -\frac{y^{2s}}{4^s \Gamma(s)} \int_0^\infty e^{-y^2/(4r)} \partial_re^{-r\lambda}\,\frac{dr}{r^{1+s}}\\
&= \frac{y^{2s}}{4^s \Gamma(s)} \int_0^\infty \partial_r \Big( \frac{e^{-y^2/(4r)}}{r^{1+s}} \Big)  e^{-r\lambda}\,dr \\
&= \frac{y^{2s}}{4^s \Gamma(s)} \int_0^\infty \Big( \partial_{yy}+\frac{1-2s}{y} \partial_y\Big)
\Big( \frac{e^{-y^2/(4r)}}{r^{1+s}} \Big)  e^{-r\lambda}\,dr \\
&= \partial_{yy}I_s(y,\lambda)+\frac{1-2s}{y}\partial_yI_s(y,\lambda).
\end{align*}
The proof of $(3)$ follows readily from the second identity in \eqref{I}
and dominated convergence. By using that the Bessel function $K_\nu$ satisfies
$$\frac{\partial}{\partial z} [z^\nu K_\nu(z) ] = - z^\nu K_{\nu-1}(z)=- z^\nu K_{1-\nu}(z)$$
we immediately obtain $(4)$. Observe that $(5.a)$ is clear from the second identity in \eqref{I}.
The asymptotic estimate (see \cite[eq.~(5.11.9)]{Lebedev})
$$K_\nu(z) = C z^{-1/2} e^{-z} \big( 1+ O(|z|^{-1})\big)\quad\hbox{as}~|z|\to\infty,~|\arg z|<\pi-\delta,~\delta>0,$$
implies $(5.b)$. To prove (5.c), observe that the function
$g(t) = e^{-\frac{y^2}{4t}\mathrm{Re}(\lambda)}t^{s-1}$ has a maximum at $t = \frac{y^ 2\mathrm{Re}(\lambda)}{4(1-s)}$
which is $g_{max}=C_s\frac{\mathrm{Re}(\lambda)^{s-1}}{y^{2-2s}}$. Hence,
$$|I_s(y,\lambda)|\leq\frac{1}{\Gamma(s)}\int_0^\infty e^{-t}g(t)\,dt\leq C_s \frac{\mathrm{Re}(\lambda)^{s-1}}{y^{2-2s}}.$$
The estimate for $\frac{1}{y}\partial_yI_s(y,\lambda)$ follows from $(4)$ and $(5.a)$.
We can bound $\partial_{yy}I_s(y,\lambda)$ by using \eqref{I_s:eq} and the previous two estimates.
We see from $(5.b)$ that $I_s(y,\lambda) \to 0$ as $y \to \infty$.  To prove $(6)$, let $J(y)$ be a smooth solution to
\eqref{I_s:eq} such that $\lim_{y\to 0^+}J(y) = 0$, $\lim_{y\to \infty} J(y) = 0$ and
$|y^{1-2s}\partial_yJ(y)|\leq C$ for all $y\geq0$.
Multiply \eqref{I_s:eq} by $y^{1-2s}\overline{J(y)}$ and integrate by parts to get
$$\int^{\infty}_0 y^{1-2s}\mathrm{Re}(\lambda)|J(y)|^2\,dy + \int^{\infty}_0 y^{1-2s} |\partial_yJ(y) |^2\,dy= 0.$$
Since $\mathrm{Re}(\lambda)>0$, it follows that $J(y)\equiv0$.
\end{proof}

\begin{rem}
The fact that Bessel functions can be used to treat extension problems was first observed in
\cite{Stinga-Torrea-CPDE}. In here we have extended \cite{Stinga-Torrea-CPDE} to apply to the case
when $\lambda$ is complex-valued. See also \cite{Stinga-Torrea-SIAM} for solutions to the extension problem
in terms of integral representations of Bessel functions for the particular case of
$(\partial_t-\Delta)^s$, in which $\lambda=i\rho+|\xi|^2$.
\end{rem}

For the sake of simplicity and concreteness of the presentation we next 
assume that $L$ is a nonnegative, normal linear operator in $L^2(\Omega)$,
with countable eigenvalues and real eigenfunctions and with a nonnegative, symmetric heat kernel, as in Section \ref{section:pointwise}.
Recall that if the first eigenvalue is $\lambda_0=0$ (as in the Neumann Laplacian)
then we assume that all the functions involved have zero spatial mean.
The general case follows by using the Spectral Theorem or the spectral
resolution of the corresponding operator (like the Fourier transform or the Hankel transform). Details are left to the interested reader.

\begin{proof}[Proof of Theorem \ref{th_exten}]
Let us denote $U(y)=U(\cdot,\cdot,y)$, for $y>0$, where $U$ is given by \eqref{solution}.
Since
\begin{equation}\label{eq:value1}
\frac{y^{2s}}{4^s\Gamma(s)}\int_0^\infty e^{-y^2/(4\tau)}\,\frac{d\tau}{\tau^{1+s}}=1
\end{equation}
we find that, for any 
$v=v(t,x)\in L^2(\mathbb{R}\times \Omega)$,
\begin{align*}
     \big|\langle U(y), v \rangle_{L^2(\R\times \Omega)} \big| &\leq \frac{y^{2s}}{4^s \Gamma(s)}
     \int^{\infty}_0 e^{-y^2/(4 \tau)} \|e^{- \tau H } u\|_{L^2(\R \times \Omega)} \norm{v}_{L^2(\R \times \Omega)}\, \frac{d \tau}{\tau^{1+s}} \\
     &\leq \norm{u}_{L^2(\R \times \Omega)}\norm{v}_{L^2(\R \times \Omega)}
\end{align*}
so that
\begin{equation}\label{eq:Uv}
\langle U(y), v \rangle_{L^2(\R\times \Omega)} = \frac{y^{2s}}{4^s \Gamma(s)} \int^{\infty}_0 e^{-y^2/(4 \tau)}
\langle e^{- \tau H } u, v \rangle_{L^2(\R \times \Omega)}\, \frac{d \tau}{\tau^{1+s}}<\infty.
\end{equation}
In particular, for each $y>0$, $U(y) \in L^2(\R \times \Omega)$, with
$$\|U(y)\|_{L^2(\R \times \Omega)} \leq \|u\|_{L^2(\R \times \Omega)}.$$
In addition, by using \eqref{eq:semigroupL2} and \eqref{I} from Lemma \ref{MainLem},
$$\langle U(y), v \rangle_{L^2(\R\times \Omega)}
 = \int_\R\sum_{k=0}^\infty\widehat{u_k}(\rho) \overline{\widehat{v_k}(\rho)} I_s(y, i \rho+ \lambda_k)\,d \rho$$
 and
$$U(y)=\frac{1}{(2\pi)^{1/2}}\int_\R\sum_{k=0}^\infty\widehat{u_k}(\rho)I_s(y, i \rho+ \lambda_k)\phi_k(x)e^{i\rho t}\,d \rho.$$

Next, by using Lemma \ref{MainLem} parts $(5.a)$ and $(5.c)$,
$$\int_\R\sum_{k=0}^\infty|i\rho+\lambda_k||\widehat{u_k}(\rho)|^2|I_s(y, i \rho+ \lambda_k)|^2\,d \rho\leq
\frac{C_s}{y^{2-2s}}\int_\R\sum_{k=0}^\infty|i\rho+\lambda_k|^s|\widehat{u_k}(\rho)|^2\,d \rho<\infty,$$
we get that $U(y)\in\Dom(H)$ for each $y>0$.
Then, for any $v\in\Dom(H)$, (see \eqref{eq:definitionHs})
$$\langle H U(y), v \rangle = \int_\R\sum_{k=0}^\infty
\widehat{u_k}(\rho) \overline{\widehat{v_k}(\rho)}(i \rho + \lambda_k) I_s(y, i \rho + \lambda_k)\,d \rho.$$

Let us check that $U \in C^\infty((0, \infty);L^2(\R\times\Omega))$ and that, for any $k\geq1$,
$$\partial_y^k\langle U(y),v\rangle_{L^2(\R\times\Omega)}=\langle\partial_y^k U(y),v\rangle_{L^2(\R\times\Omega)}.$$
Indeed, first notice that
\begin{equation}\label{eq:semigroupi}
|\langle e^{-\tau H}u, v \rangle_{L^2(\R \times \Omega)}| \leq e^{-\tau \lambda_i} \norm{u}_{L^2(\R \times \Omega)} \norm{v}_{L^2(\R \times \Omega)}
\end{equation}
where $i=0$ if $\lambda_0\neq0$ and $i=1$ if $\lambda_0=0$.
Here we have used that
$$\|e^{-\tau H}u\|_{L^2(\R\times\Omega)}^2=\sum_{k=i}^\infty e^{-2\tau\lambda_k}\int_\R|u_k(t-\tau)|^2\,dt
\leq e^{-2\tau\lambda_i}\|u\|_{L^2(\R\times\Omega)}^2.$$
By using \eqref{eq:differentiation},
\begin{multline*}
\int^{\infty}_0\bigg|\partial_y\bigg(\frac{y^{2s}}{4^s \Gamma(s)}e^{-y^2/(4 \tau)}\bigg)\langle
e^{-\tau H} u,v\rangle_{L^2(\R\times\Omega)}\bigg|\,
\frac{d \tau}{\tau^{1+s}} \\
\leq C_sy^{2s-1}\norm{u}_{L^2(\R \times \Omega)} \norm{v}_{L^2(\R \times \Omega)}
\int_0^\infty e^{-\tau\lambda_i}e^{-y^2/(c\tau)}\,\frac{d\tau}{\tau^{1+s}}
\end{multline*}
so we can differentiate under the integral sign in \eqref{eq:Uv}. Similarly it can be done
for higher order derivatives and we get $U(y)\in C^\infty((0,\infty);L^2(\R\times\Omega))$.

Observe that, by the first equation in \eqref{I},
\begin{align*}
\int^{\infty}_0 y^{1-2s}\|U\|_{H^1}^2\, dy & = \int^{\infty}_0 y^{1-2s} \int_{\R} \sum_{k=0}^\infty |i \rho + \lambda_k||\widehat{u}_k(\rho)|^2 |I_s(y, i \rho + \lambda_k) |^2 \, d \rho \, dy  \\
& =\int_{\R} \sum_{k=0}^\infty |i \rho + \lambda_k||\widehat{u}_k(\rho)|^2 \int^{\infty}_0 y^{1-2s} |I_s(y, i \rho + \lambda_k) |^2 \, dy \, d \rho\\
&\leq C_s\int_{\R} \sum_{k=0}^\infty |i \rho + \lambda_k|^{1+s}|\widehat{u}_k(\rho)|^2 \int^{\infty}_0 y
|K_s(y\sqrt{i \rho + \lambda_k})|^2 \, dy \, d \rho. 
\end{align*}
To estimate the integral in $dy$, let $r = y |\sqrt{i \rho + \lambda_k}|$ and $\theta = \arg(\sqrt{i \rho + \lambda_k})$, hence
\begin{align}\nonumber
\int^{\infty}_0 y|K_s(y\sqrt{i \rho + \lambda_k})|^2 \, dy=|i \rho + \lambda_k|^{-1} \int^{\infty}_0 r  |K_s(re^{i \theta})|^2 dr\leq
C_s|i\rho+\lambda_k|^{-1},
\end{align}
In the last inequality we used the fact that
\begin{equation}\label{eq:Ksasymptotic}
K_s(z)\sim C_sz^{-s}\quad\hbox{as}~z\to0,\quad\hbox{and}\quad K_s(z)\sim~z^{-1/2}e^{-z}\quad\hbox{as}~z\to\infty,
\end{equation}
see \cite{Lebedev}. Then,
$$ \int^{\infty}_0 y^{1-2s} \|U\|^2\,dy \leq  C_s\int_{\R} \sum_{k=0}^\infty |i \rho + \lambda_k|^s |\widehat{u}_k(\rho)|^2 \, d \rho= C_s\norm{u}^2_{H^s}<\infty$$
so $U\in L^2((0,\infty);\Dom(H),y^{1-2s}dy)$.

For $v\in \Dom(H)$, by Lemma \ref{MainLem}, we have that
\begin{align*}
 \langle HU(y),v \rangle &= \int_\R\sum_{k=0}^\infty
 \widehat{u_k}(\rho)\overline{\widehat{v_k}(\rho)}(i \rho + \lambda_k)I_s(y, i \rho + \lambda_k)\,d\rho \\
&= \int_\R\sum_{k=0}^\infty\widehat{u_k}(\rho) \overline{\widehat{v_k}(\rho)} \left( \tfrac{1-2s}{y} \partial_y + \partial_{yy} \right)
I_s(y, i \rho + \lambda_k)\,d\rho \\ 
&= \left \langle \left( \tfrac{1-2s}{y} \partial_y + \partial_{yy} \right) U(y), v \right \rangle_{L^2(\R\times\Omega)}.
\end{align*}

By Lemma \ref{MainLem} and Dominated Convergence Theorem,
$$\lim_{y\to 0}\langle U(y), v \rangle_{L^2(\R\times\Omega)}
=\int_\R\sum_{k=0}^\infty\widehat{u_k}(\rho)\overline{\widehat{v_k}(\rho)}\,d\rho
=\langle u,v \rangle_{L^2(\R\times\Omega)}$$
and
\begin{equation}\label{eq:partialyUv}
\begin{aligned}
\langle -y^{1-2s} \partial_y U(y), v \rangle_{L^2(\R\times\Omega)} &= \frac{\Gamma(1-s)}{4^{s-1/2}\Gamma(s)}\int_\R\sum_{k=0}^\infty
(i \rho + \lambda_k)^s\widehat{u_k}(\rho)\overline{\widehat{v_k}(\rho)}I_{1-s}(y, i \rho + \lambda_k)\,d \rho \\
&\to\frac{\Gamma(1-s)}{4^{s-1/2}\Gamma(s)}\langle H^su,v\rangle,\quad\hbox{as}~y\to0^+.
\end{aligned}
\end{equation}

Now, for every $v \in \Dom(H^s)$, since $I_s(0, i \rho+\lambda_k)=1$,
$$\frac{1}{y^{2s}} \langle U(y)-U(0), v\rangle_{L^2(\R \times \Omega)}
= \int_{\R} \sum_{k=0}^\infty\widehat{u_k}(\rho) \overline{\widehat{v_k}(\rho)} \frac{I_s(y, i \rho+\lambda_k)- 1}{y^{2s}} \, d \rho.$$
From the third equation in \eqref{I}, \eqref{eq:value1} and \eqref{gamma_com} we get
\begin{align*}
\frac{I_s(y, i \rho+\lambda_k)-1}{y^{2s}}&=\frac{1}{4^s\Gamma(s)}\int_0^\infty e^{-y^2/(4\tau)}\big(e^{-\tau(i\rho+\lambda_k)}-1\big)
\,\frac{d\tau}{\tau^{1+s}} \\
&\quad\to\frac{\Gamma(-s)}{4^s\Gamma(s)}(i\rho+\lambda_k)^s,\quad\hbox{as}~y\to0^+.
\end{align*}
Moreover, by applying Lemma \ref{MainLem}$(4)$ and $(5.a)$,
\begin{align*}
\frac{|I_s(y, i \rho+\lambda_k)-1|}{y^{2s}}&\leq \frac{1}{y^{2s}}\int_0^y |\partial_r I_s(r, i \rho + \lambda_k)|\,dr \\
&\leq  \frac{C_s}{y^{2s}}|i \rho + \lambda_k|^s\int^y_0 r^{2s-1}\, dr = C_s|i \rho + \lambda_k|^s.
\end{align*}
Thus, as $u,v \in \Dom(H^s)$, by Dominated Convergence Theorem,
\begin{align*}
\lim_{y\to0^+}\frac{1}{y^{2s}} \langle U(y)-U(0), v\rangle_{L^2(\R \times \Omega)} &=
\frac{\Gamma(-s)}{4^s\Gamma(s)}\int_{\R} \sum_{k=0}^\infty(i\rho + \lambda_k)^s\widehat{u_k}(\rho) \overline{\widehat{v_k}(\rho)} \, d \rho \\ 
&=\frac{\Gamma(-s)}{4^s\Gamma(s)}\langle H^s u, v \rangle.
\end{align*}

For any $v\in L^2(\R\times\Omega)$, by \eqref{eq:semigroupi} and Lemma \ref{MainLem}, we have
\begin{equation}\label{eq:lambdai}
\begin{aligned}
|\langle U(y),v\rangle_{L^2(\R\times\Omega)}| &\leq \|u\|_{L^2(\R\times\Omega)}\|v\|_{L^2(\R\times\Omega)}
\frac{y^{2s}}{4^s \Gamma(s)}\int^{\infty}_0 e^{-\tau \lambda_i}e^{- \frac{y^2}{4 \tau}}\,\frac{d \tau}{\tau^{1+s}} \\
&= \|u\|_{L^2(\R\times\Omega)}\|v\|_{L^2(\R\times\Omega)}I_s(y,\lambda_i),
\end{aligned}
\end{equation}
where $i=0$ if $\lambda_0\neq0$ and $i=1$ if $\lambda_0=0$. 
Since $I_s(y, \lambda_i) \to 0$ as $y \to \infty$, we get that $U$ weakly vanishes as $y\to\infty$.

If $v\in\Dom(H^s)$ then we see from Lemma \ref{MainLem}$(5.a)$ and \eqref{eq:partialyUv} that
$$|\langle y^{1-2s}\partial_y U,v \rangle_{L^2(\R \times \Omega)}|\leq C_s\|u\|_{H^s}\|v\|_{H^s},\quad
\hbox{for all}~y\geq0.$$
 \end{proof}
 
\section{Extension problem for parabolic operators in divergence form}\label{div_sec}

In this section we specialize the extension characterization for $(\partial_t+L)^s$ in Theorem \ref{th_exten}
to the case when $L$ is a divergence form elliptic operator.

Let $\Omega\subset\R^n$ be a (possibly unbounded) domain and
$$Lu=-\dvie(a(x) \nabla u ) + c(x) u\quad\hbox{in}~\Omega,$$
where $a(x)=(a^{ij}(x))$ is a bounded, measurable, symmetric matrix defined in $\Omega$, satisfying
the uniform ellipticity condition, that is, for some $\Lambda>0$
$$\Lambda^{-1}|\xi|^2\leq a^{ij}(x)\xi_i\xi_j\leq\Lambda|\xi|^2$$
for a.e. $x\in\Omega$, for all $\xi=(\xi_i)_{i=1}^n\in\R^n$, and $c(x) \in L^{\infty}_{\mathrm{loc}}(\Omega)$.
Let $f\in L^2(\Omega)$. For $u \in L^2(\Omega)$, $Lu = f$ in $\Omega$ in the weak sense means that
$ \nabla u \in L^2(\Omega)$, $c^{1/2} u \in L^2(\Omega)$ and 
$$\int_{\Omega} a(x) \nabla u \nabla v\,dx + \int_{\Omega} c(x) u v\,dx = \int_{\Omega} fv\,dx,$$
for every $v\in C^\infty_c(\Omega)$.
For the sake of concreteness, we assume that, under appropriate boundary conditions on $\partial\Omega$,
$L$ has a countable family of nonnegative
eigenvalues and real eigenfunctions $(\lambda_k,\phi_k)_{k=0}^\infty$
such that the set $\{\phi_k\}_{k=0}^\infty$ forms an orthonormal basis for $L^2(\Omega)$.
For more general cases when the spectrum is continuous, see Remark \ref{rem:continuousspectrumextension}.
As before, if the first eigenvalue $\lambda_0=0$ then we assume that all the functions
involved have zero spatial mean. In particular,
$$L \phi_k = \lambda_k \phi_k\quad\hbox{for all}~k\geq0~\hbox{in the weak sense}.$$  
Therefore, if we define
$$H^1_{L}(\Omega)\equiv \Dom(L)=\Big\{u\in L^2(\Omega):\sum_{k=0}^\infty\lambda_k|u_k|^2<\infty\Big\}$$
where $\displaystyle u_k=\int_{\Omega}u\phi_k\,dx$, then, for any $u,v\in H^1_L(\Omega)$,
$$\int_{\Omega} a(x) \nabla u \nabla v\,dx + \int_{\Omega} c(x) u v\,dx=\sum_{k=0}^\infty\lambda_k u_kv_k.$$
The operators listed in $(1)$--$(4)$ in the Introduction satisfy the conditions above.

Now, the extension equation takes the form 
$$\partial_t U=y^{-(1-2s)}\dvie_{x,y}(y^{1-2s}B(x)\nabla_{x,y}U)-c(x)U,$$
where
$$B(x)=
\begin{bmatrix}
a(x) & 0 \\
0 & 1
\end{bmatrix}$$
is also uniformly elliptic. Let us denote $D=\{(x,y):x\in\Omega,~y>0\}\subset\R^{n+1}$.
The weight $\omega(x,y)=|y|^{1-2s}$ belongs to the Muckenhoupt class $A_2(\R^{n+1})$.
Define $H^1_{L,y}(D)$ as the set of functions $w=w(x,y)\in L^2(D,y^{1-2s}dxdy)$ such that 
\begin{align*}
[w]_{H^1_{L,y}(D)}^2 &:= \int^{\infty}_0\int_\Omega y^{1-2s}\big(a(x)\nabla w\nabla w+c(x)w^2\big)
\,dx\,dy + \int^{\infty}_0 \int_{\Omega} y^{1-2s}
|\partial_yw|^2\,dx\,dy\\
&=\int^{\infty}_0y^{1-2s}\sum_{k=0}^\infty\lambda_k|w_k(y)|^2\,dy + \int^{\infty}_0 \int_{\Omega} y^{1-2s}
|\partial_yw|^2\,dx\,dy< \infty,
\end{align*}
where $\displaystyle w_k(y)=\int_{\Omega}w(x,y)\phi_k(x)\,dx$, under the norm
$$\|w\|_{H^1_{L,y}(D)}^2=\|w\|_{L^2(D,y^{1-2s}dxdy)}^2+[w]_{H^1_{L,y}(D)}^2.$$

\begin{thm}\label{exten_divergence}
Consider the extension problem in Theorem \ref{th_exten} with $L$ is as above.
Then $U$, defined in \eqref{solution},  belongs to $L^2(\R;H^1_{L,y}(D))\cap C^\infty((0,\infty);L^2(\R\times\Omega))\cap C([0,\infty);L^2(\R \times \Omega))$ and
for any fixed $y>0$ and $v\in C^\infty_c(\R\times\Omega)$,
$$\langle HU,v\rangle=\int_\R\int_\Omega\Big(\tfrac{1-2s}{y}\partial_y+\partial_{yy}\Big)Uv\,dt \,dx
=y^{2s-1}\int_\R\int_\Omega \partial_y(y^{1-2s}\partial_yU)v\,dt\,dx.$$
In particular, $U$ is a weak solution to the parabolic extension problem
$$\begin{cases}
\partial_t U=y^{-(1-2s)}\dvie_{x,y}(y^{1-2s}B(x)\nabla_{x,y}U)-c(x)U&\hbox{for}~(t,x,y)\in\R\times\Omega\times(0,\infty)\\
-y^{1-2s}\partial_yU\Big|_{y=0^+}=\frac{\Gamma(1-s)}{4^{s-1/2}\Gamma(s)}H^su&\hbox{for}~(t,x)\in\R\times\Omega
\end{cases}$$
in the following sense: for any $V(t,x,y) \in C^{\infty}_c (\R\times\Omega\times[0, \infty) )$,
\begin{equation}\label{eq:Ubeforeintegration}
\begin{aligned}
\int_\R\int_\Omega U\partial_t V\,dx\,dt &= \int_\R\int_\Omega \big(a(x)\nabla_xU\nabla_xV+c(x)UV\big)\,dx \,dt \\
&\quad-\int_\R\int_\Omega\Big(\tfrac{1-2s}{y}\partial_y+\partial_{yy}\Big)UV \,dx\,dt
\end{aligned}
\end{equation}
and
\begin{align*}
\int_0^\infty\int_\R\int_\Omega y^{1-2s}U\partial_tV\,dx \,dt \,dy 
&=\int_0^\infty\int_\R\int_\Omega y^{1-2s}\big(B(x)\nabla_{x,y}U\nabla_{x,y}V+c(x)UV\big)\,dx\,dt \,dy \\
&\quad-\frac{\Gamma(1-s)}{4^{s-1/2}\Gamma(s)}\langle H^su, V(t,x,0) \rangle.
\end{align*}
\end{thm}

\begin{proof}
We have already proved in the general extension result, Theorem \ref{th_exten},
that $U(\cdot,\cdot,y)\in C^\infty((0,\infty);L^2(\R\times\Omega))\cap C([0,\infty);L^2(\R \times \Omega))$.

Let us next check that $U(t,x,y)\in L^2(\R;H^1_{L,y}(D))$. We found in \eqref{eq:lambdai} that
$$\|U(y)\|_{L^2(\R\times\Omega)} \leq \|u\|_{L^2(\R\times\Omega)}I_s(y,\lambda_i)$$
where $i=0$ if $\lambda_0\neq0$ and $i=1$ if $\lambda_0=0$. Then, from \eqref{I},
\begin{equation}\label{eq:computationwithKs}
\begin{aligned}
\int^{\infty}_0y^{1-2s}\|U(y)\|^2_{L^2(\R\times\Omega)}\,dy &\le C_s\|u\|^2_{L^2(\R\times\Omega)}
\int^{\infty}_0y^{1-2s}(y\sqrt{\lambda_i})^{2s}K^2_s(y\sqrt{\lambda_i})\,dy \\
&= C_s\|u\|^2_{L^2(\R\times\Omega)}\lambda_i^{s-1}\int^{\infty}_0 rK^2_s(r) dr<\infty.
\end{aligned}
\end{equation}
In the last inequality we used \eqref{eq:Ksasymptotic}. We are left to show that
$$\int_\R \int_0^\infty y^{1-2s}\sum_{k=0}^\infty \lambda_k |U_k(t,y)|^2\,dy\,dt + \int_\R \int_0^\infty\int_\Omega y^{1-2s}
|\partial_y U(t,x,y)|^2  \,dx \,dy\,dt < \infty,$$
 where, for any $k\geq i$, for $i=0$ if $\lambda_0\neq0$ and $i=1$ if $\lambda_0=0$,
\begin{align*}
U_k(t,y) &= \langle U(t,\cdot,y),\phi_k(\cdot)\rangle_{L^2(\Omega)} \\
&= \frac{y^{2s}}{4^s \Gamma(s)} \int^{\infty}_0 e^{-y^2/(4 \tau)}\langle e^{-\tau L}u(t-\tau,\cdot),\phi_k(\cdot)\rangle_{L^2(\Omega)}
\,\frac{d\tau}{\tau^{1+s}} \\
&= \frac{y^{2s}}{4^s \Gamma(s)} \int^{\infty}_0 e^{-y^2/(4 \tau)}e^{-\tau\lambda_k}u_k(t-\tau)\,\frac{d \tau}{\tau^{1+s}}.
\end{align*}
From here and \eqref{I} we see that
$$\int_\R |U_k(t,y)|^2\,dt \leq  \|u_k\|^2_{L^2(\R)}|I_s(y, \lambda_k)|^2.$$
Therefore, as done in \eqref{eq:computationwithKs}, 
\begin{align*}
\int_\R \int_0^\infty y^{1-2s}\sum_{k=0}^\infty \lambda_k |U_k(t,y)|^2\,dy \,dt & \leq  \sum_{k=0}^\infty
\lambda_k\|u_k\|^2_{L^2(\R)}\int_0^\infty y^{1-2s} (y\sqrt{\lambda_k})^{2s}K^2_s(y\sqrt{\lambda_k})\,dy \\
&\le C_s\sum_{k=0}^\infty \lambda^s_k\|u_k\|^2_{L^2(\R)} < \infty.
\end{align*} 
Next, observe that
$$\partial_y U(t,x,y) =C_sy^{2s-1}\sum_{k=0}^\infty\bigg[\int_{\R}
\widehat{u_k}(\rho)(i \rho+ \lambda_k)^s  I_{1-s}(y, i\rho+ \lambda_k) e^{i \rho t} d \rho \bigg] \phi_k(x)$$
and then
$$\|\partial_y U\|^2_{L^2(\R\times\Omega)}
= C_sy^{2s}\sum_{k=0}^\infty\int_{\R} |\widehat{u_k}(\rho)|^2|i \rho + \lambda_k|^{1+s}|K_{1-s}(y \sqrt{i \rho + \lambda_k})|^2\,d \rho. $$
Hence,
\begin{multline*}
\int_0^\infty y^{1-2s}\|\partial_y U\|^2_{L^2(\R\times\Omega)}\,dy\\
=C_s\sum_{k=0}^\infty \int_\R |\widehat{u_k}(\rho)|^2 |i \rho + \lambda_k|^{1+s} \int_0^\infty y|K_{1-s}(y \sqrt{i \rho + \lambda_k})|^2\,dy\,d\rho.
\end{multline*}
To estimate the integral in $dy$, we write $r = y |\sqrt{i \rho + \lambda_k}|$ and $\theta=\arg \left( \sqrt{i \rho + \lambda_k}  \right)$ to get
$$\int^{\infty}_0 y |K_{1-s}(y \sqrt{i \rho + \lambda_k})|^2\,dy=
\frac{1}{|i \rho + \lambda_k|} \int^{\infty}_0 r|K_{1-s}( r e^{i \theta})|^2\,dr\leq\frac{C_s}{|i\rho+\lambda_k|},$$
because of \eqref{eq:Ksasymptotic}. Whence,
$$\int_0^\infty y^{1-2s}\|\partial_y U\|^2_{L^2(\R\times\Omega)}\,dy\leq C_s
\sum_{k=0}^\infty\int_\R |\widehat{u_k}(\rho)|^2 |i \rho + \lambda_k|^{s} d\rho < \infty.$$
Thus $U(t,x,y)\in L^2(\R;H^1_{L,y}(D))$, as desired.

Let $V\in C^{\infty}_c(\R\times\Omega\times[0, \infty) )$. The action of $\partial_tU$ on $V$ is given by
$$\partial_tU(V)=-\int_{\R}U\partial_tVdt$$
for a.e.~$(x,y)\in\Omega\times[0,\infty)$. 
For a fixed $y$, we already know that
$$\langle HU,V\rangle=\int_\R\int_\Omega\Big(\tfrac{1-2s}{y}\partial_y+\partial_{yy}\Big)UV\,dt \,dx
=y^{2s-1}\int_\R\int_\Omega \partial_y(y^{1-2s}\partial_yU)V\,dt\,dx.$$
But now,
\begin{align*}
\langle HU,V\rangle &= - \int_{\R} \sum_{k=0}^\infty \widehat{u_k}(\rho) I_s(y, i \rho + \lambda_k)\overline{i \rho \widehat{V_k}(\rho,y)} \,d \rho \\
&\quad+ \int_{\R} \sum_{k=0}^\infty \lambda_k\widehat{u_k}(\rho)I_s(y, i \rho + \lambda_k) \overline{\widehat{V_k}(\rho,y)}  \,d \rho \\
&= - \int_{\R} \sum_{k=0}^\infty \widehat{u_k}(\rho) I_s(y, i \rho + \lambda_k)\overline{\widehat{\partial_tV_k}(\rho,y)}\, d \rho\\
&\quad+ \int_{\R} \sum_{k=0}^\infty \lambda_k\widehat{u_k}(\rho)I_s(y, i \rho + \lambda_k) \overline{\widehat{V_k}(\rho,y)}  \,d \rho \\
&= - \int_\R\int_\Omega U\partial_t V\,dx\,dt+\int_\R\int_\Omega \big(a(x)\nabla_xU\nabla_xV+c(x)UV\big)\,dx \,dt.
\end{align*}
Thus, \eqref{eq:Ubeforeintegration} follows.

Let us multiply \eqref{eq:Ubeforeintegration} by $y^{1-2s}$ and integrate in $dy$ to obtain
\begin{align*}
\int_0^\infty\int_\R\int_\Omega y^{1-2s}\partial_tU(V)\,dx\,dt\,dy &=
-\int_0^\infty\int_\R\int_\Omega y^{1-2s}\big(a(x)\nabla_xU\nabla_xV +c(x)UV\big)\,dx \,dt \,dy \\ 
&\quad+\int_0^\infty \int_\R\int_\Omega y^{1-2s}\Big(\tfrac{1-2s}{y} \partial_y + \partial_{yy} \Big)UV\,dx  \,dt\,dy.
\end{align*}
Let $0<a<b<\infty$. Since $U\in C^\infty((0,\infty);L^2(\R\times\Omega))$ we can apply Fubini's Theorem
and integration by parts to get
\begin{align*}
\int_a^b\int_\R&\int_\Omega y^{1-2s}\Big(\tfrac{1-2s}{y}\partial_y+\partial_{yy}\Big)UV\,dx\,dt\,dy \\
&= \int_\R\int_\Omega\int_a^b\partial_y(y^{1-2s}\partial_yU)V\,dy\,dt\,dx \\
&= -\int_a^b \int_\R\int_\Omega y^{1-2s}\partial_y U \partial_y V \,dy \,dx \,dt +
 \int_{\R} \int_{\Omega} y^{1-2s}\partial_yUV\,dx \,dt\big|_{y=a}^{y=b}.
\end{align*}
By letting $a\to0$ and $b\to\infty$, we have
\begin{align*}
\int_0^\infty\int_\R&\int_\Omega y^{1-2s}\Big(\tfrac{1-2s}{y}\partial_y+\partial_{yy}\Big)UV\,dx\,dt\,dy \\
&= -\int_0^\infty\int_\R\int_\Omega y^{1-2s}\partial_y U \partial_y V \,dy \,dx \,dt -
 \lim_{y\to0^+}\int_{\R} \int_{\Omega} y^{1-2s}\partial_yUV\,dx \,dt.
\end{align*}
To conclude,
\begin{align*}
\lim_{y \rightarrow 0}
\int_{\R} \int_{\Omega}\big(y^{1-2s}\partial_yUV\big)\,dx \,dt
&= \lim_{y \rightarrow 0} \int_{\R} \int_{\Omega}y^{1-2s}\partial_y U\big(V(t,x,y)-V(t,x,0)\big)\,dx \,dt \\
&\quad+ \lim_{y \rightarrow 0} \int_{\R} \int_{\Omega}y^{1-2s}\partial_yUV(t,x,0)\,dx \,dt \\
&= 0-\frac{\Gamma(1-s)}{4^{s-1/2}\Gamma(s)}\langle H^su,V(\cdot,\cdot,0)\rangle,
\end{align*}
where for the last identity we have used \eqref{eq:partialyUv}, the fact that
$V\in C^\infty_c(\R\times\Omega\times[0,\infty))$ and Dominated Convergence Theorem. Indeed,
\begin{align*}
\bigg|&\int_{\R}\int_{\Omega}y^{1-2s}\partial_y U\big(V(t,x,y)-V(t,x,0)\big)\,dx \,dt\bigg|^2
\leq C_s\|u\|_{H^s}^2\|V(\cdot,\cdot,y)-V(\cdot,\cdot,0)\|_{H^s}^2 \\
&\leq C_s\|u\|_{H^s}^2\|V(\cdot,\cdot,y)-V(\cdot,\cdot,0)\|_{H^1}^2 \\
&\leq C_{s,\Lambda}\|u\|_{H^s}^2\bigg\{\|V(\cdot,\cdot,y)-V(\cdot,\cdot,0)\|_{L^2(\R\times\Omega)}^2+\int_{\R}\int_\Omega|\partial_t(V(t,x,y)-V(t,x,0))|^2\,dx\,dt\\&
+\int_{\R}\int_\Omega|\nabla_x(V(t,x,y)-V(t,x,0))|^2\,dx\,dt+\int_{\R}\int_\Omega|c(x)||V(t,x,y)-V(t,x,0)|^2\,dx\,dt\bigg\} \\
&\to 0\quad\hbox{as}~y\to0.
\end{align*}
\end{proof}

\begin{lem}[Reflection extension]\label{reflexionext}
Let $L$ and $U$ be as in Theorem \ref{exten_divergence}.
Let $\Omega_0 \subset \Omega$ be a bounded domain and
$(T_0,T_1)\subset\R$. Suppose that
$$\lim_{y\to 0^+}\langle y^{1-2s}\partial_yU,V\rangle_{L^2(\R\times\Omega)}=0$$
for all $V\in C^\infty_c((T_0,T_1)\times\Omega_0\times[0,\infty))$. 
Fix $Y_0>0$. 
Then, the even extension $\widetilde{U}$ of $U$ in the variable $y$, defined by
\begin{equation}
\label{reflex}
\widetilde{U}(t,x,y) =\begin{cases}
U(t,x,y)&\hbox{for}~0\le y< Y_0 \\
U(t,x,-y)&\hbox{for}~-Y_0<y<0
\end{cases}\end{equation}
is a weak solution to the degenerate parabolic equation
\begin{equation}\label{eq:reflec_exten}
\partial_t\widetilde{U}=|y|^{-(1-2s)}\dvie_{x,y}(|y|^{1-2s}B(x)\nabla_{x,y}\widetilde{U}) - c(x)\widetilde{U}
\end{equation}
in $(T_0,T_1)\times\Omega_0\times(-Y_0,Y_0)$.
\end{lem}

\begin{proof}
Let $V \in C^{\infty}_c((T_1,T_2) \times \Omega_0 \times (-Y_0, Y_0))$. We shall prove that
\begin{multline*}
\int_{T_0}^{T_1}\int_{-Y_0}^{Y_0}\int_{\Omega_0}|y|^{1-2s}\widetilde{U}\partial_tV\,dx\,dy\,dt \\
=\int_{T_0}^{T_1}\int_{-Y_0}^{Y_0}\int_{\Omega_0}|y|^{1-2s}\big(B(x)\nabla_{x,y}\widetilde{U}\nabla_{x,y}V+c(x)\widetilde{U}V\big)\,dx\,dy\,dt.
\end{multline*}
Let $\delta>0$. From \eqref{eq:Ubeforeintegration}, for any $y>0$,
\begin{align*}
\int_\R\int_\Omega U\partial_t V\,dx\,dt &= \int_\R\int_\Omega \big(a(x)\nabla_xU\nabla_xV+c(x)UV\big)\,dx \,dt \\
&\quad-\int_\R\int_\Omega|y|^{2s-1}\partial_y(|y|^{1-2s}\partial_yU)V \,dx\,dt.
\end{align*}
By multiplying this equation by $|y|^{1-2s}$, integrating in $y\in(\delta,Y_0)$, and using integration by parts we get
\begin{multline*}
\int_{T_0}^{T_1}\int_{\delta}^{Y_0}\int_{\Omega_0}|y|^{1-2s}\widetilde{U}\partial_tV\,dx\,dy\,dt \\
=\int_{T_0}^{T_1}\int_{\delta}^{Y_0}\int_{\Omega_0}|y|^{1-2s}\big(B(x)\nabla_{x,y}\widetilde{U}\nabla_{x,y}V+c(x)\widetilde{U}V\big)\,dx\,dy\,dt\\
+\int_{T_0}^{T_1}\int_{\Omega_0}\delta^{1-2s}\partial_yU(t,x,\delta)V(t,x,\delta)\,dx\,dt.
\end{multline*}
From here we readily get
\begin{multline*}
\int_{T_0}^{T_1}\int_{\delta<|y|<Y_0}\int_{\Omega_0}|y|^{1-2s}\widetilde{U}\partial_tV\,dx\,dy\,dt \\
=\int_{T_0}^{T_1}\int_{\delta<|y|<Y_0}\int_{\Omega_0}|y|^{1-2s}\big(B(x)\nabla_{x,y}\widetilde{U}
\nabla_{x,y}V+c(x)\widetilde{U}V\big)\,dx\,dy\,dt\\
+\int_{T_0}^{T_1}\int_{\Omega_0}\delta^{1-2s}\partial_yU(t,x,y)|_{y=\delta}V(t,x,-\delta)\,dx\,dt\\
+\int_{T_0}^{T_1}\int_{\Omega_0}\delta^{1-2s}\partial_yU(t,x,\delta)V(t,x,\delta)\,dx\,dt.
\end{multline*}
The conclusion follows by taking $\delta\to0$ in this last identity.
\end{proof}

\begin{rem}\label{rem:continuousspectrumextension}
If the elliptic operator $L$ has continuous spectrum,
then all the previous results are still valid. Indeed, one needs to use the corresponding spectral resolution.

Consider, for example, $L=-\Delta$ in $\Omega=\R^n$.
We can use Fourier transform $\mathcal{F}$ in the variables $t$ and $x$ to define the operator $(\partial_t + L)^s$ as
$$\langle (\partial_t - \Delta)^s u, v \rangle_{L^2(\R^{n+1})}= \int_{\R} \int_{\R^n}(i \rho + |\xi|^2)^s \mathcal{F}u(\rho, \xi) 
\overline{\mathcal{F}v(\rho, \xi)}\, d\xi \, d \rho.$$
The analogous to the expression
$$u(t,x) = \sum_{k=0}^\infty u_k(t)\phi_k(x)$$
in this case is just
$$ u(t,x) =\frac{1}{(2\pi)^{n/2}}\int_{\R^n} \widehat{u}(t,\xi)e^{i \xi\cdot x}\,d \xi$$
where the Fourier transform is taken in the $x$ variable by leaving $t$ fixed. 
The eigenvalues and eigenfunctions $(\lambda_k, \phi_k)_{k=0}^\infty$ are replaced by
$(|\xi|^2, e^{i x\cdot\xi})_{\xi\in\R^n}$. 

Consider another example, the Bessel operator $L = -\frac{d^2}{dx^2} + \frac{\lambda(\lambda-1)}{x^2}$, for $\lambda >0$, in $\Omega = (0, \infty)$. In this case we can use Hankel transform in $x$ and Fourier transform in $t$. Let
$\phi_y(x) = (yx)^{1/2} J_{\lambda-1/2}(yx) $, $x,y>0$, where $J_{\nu}$ denotes the Bessel function of the first kind with order $\nu$.
Then $L\phi_y(x) = y^2 \phi_y(x) $ and the eigenvalues and eigenfunctions $(\lambda_k,\phi_k)_{k=0}^\infty$ 
are replaced by $(y^2,\phi_y(x))_{y>0}$. The Hankel transform in the variable $x$ is defined as 
$$\mathcal{H}u(t,y) = \int^{\infty}_0 u(t,x) \phi_{y}(x)\,dx $$
and, since $\mathcal{H}^{-1}=\mathcal{H}$, we can write 
$$u(t,x) = \int^{\infty}_0 \mathcal{H}u(t,y) \phi_y(x)\,dy.$$
With this, we can let
$$\langle (\partial_t + L)^s u , v \rangle = \int_{\R} \int^{\infty}_{0} (i \rho + y^2)^s \mathcal{H}\widehat{u}(\rho,y)
\overline{\mathcal{H}\widehat{v}(\rho,y)}\,dy\, d \rho.$$

Similarly, Lemma \ref{reflexionext} holds in all these cases.
\end{rem}

\section{Proof of Theorems \ref{harn_para}  and \ref{harn_bd}}\label{proof_harn_para} 

\begin{proof}[Proof of Theorem \ref{harn_para}]
Consider the extension $U$ of $u$ given by Theorems \ref{th_exten} and \ref{exten_divergence}.
If $u\geq0$ in $(-\infty,1)\times\Omega$ then, since the heat kernel for $L$ is nonnegative, the first formula in
\eqref{solution} gives that $U\geq 0$ in $(0,1) \times B_{2r} \times [0,2)$.  
Lemma \ref{reflexionext} with $Y_0=2$ implies that $\widetilde{U}$, as defined by \eqref{reflex},
 is a nonnegative weak solution to \eqref{eq:reflec_exten}
in $(t,x,y)\in(0,1)\times B_{2r} \times(-2,2)$. The parabolic Harnack inequality due to Ishige \cite{Ish} gives
the existence of a constant $C_H>0$ such that 
\begin{align*}
\sup_{R^-} u(t,x) &= \sup_{R^-} \widetilde{U}(t,x,0)  \leq \sup_{R^- \times (-1,1)} \widetilde{U}(t,x,y)\\
&\leq C_H \inf_{R^+ \times (-1,1)} \widetilde{U}(t,x,y) \\
&\leq C_H \inf_{R^+}\widetilde{U}(t,x,0) = C_H \inf_{R^+} u(t,x).
\end{align*}

 Now we prove the local boundedness and H\"older estimates on $u$. By using the results in \cite{Ish} we get that
 $\widetilde{U}$ is locally bounded and locally parabolically Holder continuous
of order $0<\alpha<1$ in $R$. Let $K$ be a compact subset of $R$. We have
$$\|\widetilde{U}\|_{L^{\infty}(K\times (-1,1))} \leq C\|\widetilde{U}\|_{L^2(R\times (-2,2))}=2C\|U\|_{L^2(R\times (0,2))}.$$
Since $\|U\|_{L^2(R\times (0,2))}\leq C\|u\|_{L^2(\R \times \Omega)}$, we obtain
$$\|u\|_{L^{\infty}(K)} \leq \|\widetilde{U}\|_{L^{\infty}(K\times (-1,1))} \leq C \|u\|_{L^2(\R \times \Omega)}.$$
Next, from the local H\"older continuity of $\widetilde{U}$, 
$$[u]_{C^{\alpha/2,\alpha}_{t,x}(K)}=[\widetilde{U}]_{C^{\alpha/2,\alpha}_{t,x}(K\cap\{y=0\})}
\leq C\|\widetilde{U}\|_{L^{\infty}(K\times (-1,1))}\leq C\|u\|_{L^2(\R \times \Omega)}.$$
\end{proof}

\begin{rem}\label{Harnackint}
If in Theorem \ref{harn_para}  we substitute $B_{2r}$ by an open set and $B_{r}$ by a compact set contained in the open set,
the result remains valid and the constant $c$ also depends on both sets.
\end{rem}

\begin{proof}[Proof of Theorem \ref{harn_bd}]
For simplicity, and without loss of generality, we will assume that $\tilde{x}=0$.
Let $\widetilde{U}$ be the reflection in $y$ of the extension $U$ of $u$.
By Lemma \ref{reflexionext}, $\widetilde{U}$ is a nonnegative weak solution to \eqref{eq:reflec_exten} in
$(t,x,y)\in(-2,2)\times(B_{2r}(0)\cap \Omega_0)\times(-2r,2r)$ that vanishes continuously in
$(t,x,y)\in(-2,2) \times ((\Omega \setminus \Omega_0) \cap B_{2r}(0)) \times \{0 \}$.

As a first step we flatten the boundary of $\Omega_0$ inside $B_{2r}(0)$. We use a bi-Lipschitz transformation $\Psi$
such that $\Psi(0)=0$ and $\Psi(\Omega_0\cap B_{2r}(0))=\Omega_1$, where $\Omega_1$ is a new domain with flat
boundary at $x_n=0$, which can be extended as constant in $t$ and $y$.
Without loss of generality we can assume that the flat part of $B_{2r}(0)\cap \R^n_{+}$
is the flat part of the new domain $\Omega_1$. Then the transformed function $\widetilde{U}_1:=\widetilde{U}\circ\Psi^{-1}$
satisfies the same type of degenerate parabolic equation with bounded measurable coefficients in the domain
$(-2,2)\times (\R^n_{+}\cap B_{2r}(0))\times(-2r,2r)$ and vanishes continuously on
$(-2,2)\times((\R^n \setminus\R^n_{+})\cap B_{2r}(0))\times \{0\}$.

As a second step, we define a transformation which maps
$\R^{n+1} \setminus \{x_n \leq 0 , y=0 \}$ into $\R^{n+1} \cap \{x_n >0 \}$
and is extended to be constant in $t$. This construction is standard, see \cite{Stinga-Torrea-SIAM}.
After this transformation is performed, we obtain a function $\widetilde{U}_2$ that 
solves again a degenerate parabolic equation with bounded measurable coefficients
in the domain $(-2,2) \times (\R^n_{+} \cap B_{2r}(0)) \times(-2r,2r)$ and that vanishes continuously for
$(t,x,y)\in(-2,2)\times\{(x',0,y):(x')^2+y^2<(2r)^2\}$.

Now we can apply the boundary Harnack inequality of Ishige \cite{Ish} to $\widetilde{U}_2$ to get 
$$ \sup_{(-1,1) \times (\Omega \cap B_r(0))} u(t,x)=\sup_{(-1,1) \times (\R^n_{+}\cap B_r(0))} \widetilde{U}_2(t,x,0)
\leq C \widetilde{U}_2(t_0,\widetilde{x}_0,0) = u(t_0,x_0),$$
where $\widetilde{x}_0$ is the point obtained from $x_0$ via the two transformations.
\end{proof}

\begin{rem}\label{Harnackbounda}
If in Theorem \ref{harn_bd}  we substitute $B_{2r}(\tilde{x})$ by an open set and $B_{r}(\tilde{x})$ by another open subset
of the first one, the result remains still valid and the constant $C$ also depends on both open sets.
\end{rem}

\section{Transference Method}\label{trans_sec}

In this section we assume that
$$Lu=-\dvie(a(x)\nabla u)+c(x)u\quad\hbox{in}~\Omega$$
is an operator as in Section \ref{div_sec}.

\subsection{Change of variables}

Let $\tilde{\Omega}\subset \R^n$ be a domain and $h : \Omega \rightarrow \widetilde{\Omega}$ be a smooth change of variables from
$x\in\Omega$ into $\widetilde{x}=h(x)\in\widetilde{\Omega}$, 
that is, $h$ is one-to-one, onto and differentiable with inverse $h^{-1}:\widetilde{\Omega}\to\Omega$ differentiable as well.
We denote by $J_h(x)=|\det\nabla h(x)|$, for $x\in\Omega$, and $J_{h^{-1}}(\widetilde{x})=|\det\nabla h^{-1}(\widetilde{x})|$,
for $\widetilde{x}\in\widetilde{\Omega}$.
Let us define the change of variables application
$$W: L^2(\widetilde{\Omega}, J_{h^{-1}}d \widetilde{x} )\to L^2(\Omega, dx)$$
as
$$W(\widetilde{f})(x)=\widetilde{f}(h(x))\quad\hbox{for}~x\in\Omega.$$
Then $W$ is one-to-one, onto and, for any $f\in L^2(\Omega,dx)$,
$$W^{-1}(f)(\widetilde{x}) = f ( h^{-1}(\widetilde{x})),\quad~\widetilde{x}\in\widetilde{\Omega}.$$
 It is readily seen that
$$\|W\widetilde{f}\|_{L^2(\Omega,dx)}=\|\widetilde{f}\|_{L^2(\widetilde{\Omega},J_{h^{-1}}d\widetilde{x})}.$$

Let $\{\phi_k\}_{k=0}^\infty$ be the orthonormal basis of $L^2(\Omega,dx)$ consisting of
eigenfunctions of $L$. We claim that $\{\widetilde{\phi}_k:=W^{-1}\phi_k\}_{k=0}^\infty$ is
an orthonormal basis of $L^2(\widetilde{\Omega}, J_{h^{-1}}d\widetilde{x})$. Indeed,
by changing variables,
$$\int_{\widetilde{\Omega}}\widetilde{\phi}_k(\widetilde{x}) \widetilde{\phi}_\ell(\widetilde{x})J_{h^{-1}}(\widetilde{x})
\,d\widetilde{x}= \int_{\Omega} \phi_k(x)\phi_\ell(x)\,dx = \delta_{k\ell}.$$
Also, if $\widetilde{f}\in L^2(\widetilde{\Omega},J_{h^{-1}}d\widetilde{x})$
is orthogonal to each $\widetilde{\phi}_k$ then
$$0=\int_{\widetilde{\Omega}} \widetilde{f}(\widetilde{x})\widetilde{\phi}_k(\widetilde{x})
J_{h^{-1}}(\widetilde{x})\,d\widetilde{x}=\int_{\Omega} W(\widetilde{f})(x)\phi_k (x)\,dx$$
for all $k\geq0$, which gives $\widetilde{f}=0$, and the orthonormal set 
$\{\widetilde{\phi}_k\}_{k=0}^\infty$ is complete in $L^2(\widetilde{\Omega}, J_{h^{-1}}d\widetilde{x})$.

If $u\in\Dom(L)$ and we define $\widetilde{u}=W^{-1}u=u\circ h^{-1}$ then we can write
$u=W\widetilde{u}=\widetilde{u}\circ h$ and the change rule gives
$$u_{x_i}(x) = \sum^n_{k=1} \widetilde{u}_{\widetilde{x}_k}(h(x))(\nabla h(x))_{ki}$$
where $(\nabla h(x))_{ki}=\big(\frac{\partial h_k(x)}{\partial x_i} \big)_{ki}$ denotes the $ki$-th entry of the matrix $\nabla h(x)$.
From the definition of the action of $L$ on $u$ we have, for any $v\in\Dom(L)$,
 \begin{align*} \nonumber
 \langle Lu, v \rangle &= \int_{\Omega}\Big(\sum_{i,j=1}^na^{ij}(x)u_{x_i}(x) {{v_{x_j}(x)}}+c(x) u(x) {{v(x)}}\Big) \, dx \\
 &= \int_{\Omega}\bigg[\sum_{k,\ell=1}^n\Big(\sum_{i,j=1}^na^{ij}(x)(\nabla h(x))_{ki}(\nabla h(x))_{\ell j}\Big)\widetilde{u}_{\widetilde{x}_k}(h(x))
 {{\widetilde{v}_{\widetilde{x}_\ell}(h(x))}}+ c(x) u(x) {{v(x)}}\bigg]\, dx \\
 &= \int_{\widetilde{\Omega}}\big(\widetilde{a}(\widetilde{x})\nabla\widetilde{u}{{\nabla\widetilde{v}}}
 +\widetilde{c}(\widetilde{x})\widetilde{u}{{\widetilde{v}}}\big)J_{h^{-1}}(\widetilde{x})\,d\widetilde{x}
 \end{align*}
 where
 $$\widetilde{a}^{kl}(\widetilde{x}) = \sum_{i,j=1}^n a^{ij}(h^{-1}(\widetilde{x})) (\nabla h(h^{-1}(\widetilde{x})))_{ki}
 (\nabla h(h^{-1}(\widetilde{x})))_{\ell j}$$
 and
 $$\widetilde{c}(\widetilde{x})=c(h^{-1}(\widetilde{x})).$$
 With this identity we define a new operator $\widetilde{L}$ in the following way.
 Let $\widetilde{u}, \widetilde{v} \in L^2(\widetilde{\Omega},J_{h^{-1}}d\widetilde{x})$
 such that $u=W\widetilde{u}$ and $v=W\widetilde{v}$ belong to $\Dom(L)$.
We define
$$\langle\widetilde{L}\widetilde{u},\widetilde{v} \rangle := \langle Lu, v \rangle. $$ 
With this, $(\lambda_k,\widetilde{\phi}_k)_{k=0}^\infty$ are the eigenvalues and eigenfunctions of $\widetilde{L}$,
where $\lambda_k$ are the eigenvalues of $L$. Moreover,
$$\Dom(\widetilde{L}) =\Big\{\widetilde{u} \in L^2(\widetilde{\Omega},J_{h^{-1}}d\widetilde{x}):
\sum_{k=0}^\infty \lambda_k \widetilde{u}^2_k < \infty\Big\},$$
where $\displaystyle\widetilde{u}_k=\int_{\widetilde{\Omega}}\widetilde{u}\widetilde{\phi}_kJ_{h^{-1}}(\widetilde{x})\,d\widetilde{x}$.
We also notice that, if $\widetilde{u} \in L^2(\widetilde{\Omega},J_{h^{-1}}d\widetilde{x})$ and $v\in L^2(\Omega,dx)$ then 
$$ \int_{\Omega} (W \widetilde{u})(x){{v(x)}}\,dx = \int_{\widetilde{\Omega}} \widetilde{u}(\widetilde{x})(W^{-1}{{v}})(\widetilde{x})J_{h^{-1}}
(\widetilde{x})\,d\widetilde{x}.$$
Then we can formally write
$$\langle \widetilde{L} \widetilde{u}, \widetilde{v} \rangle = \langle L(W \widetilde{u}), (W\widetilde{v}) \rangle
= \langle W^{-1}LW\widetilde{u},\widetilde{v} \rangle,$$
or
$$\widetilde{L} =W^{-1}\circ L\circ W.$$
 
\subsection{Multiplication operator}

Let $M=M(x)\in C^{\infty}(\Omega)$ be a positive function.
We define the multiplication operator
$$U: L^2(\Omega, M(x)^2 dx)\to L^2(\Omega,  dx)$$
as
$$ U(\breve{u})(x)=M(x)\breve{u}(x),$$
for $\breve{u}\in L^2(\Omega, M(x)^2 dx)$. If $\{\phi_k\}_{k=0}^\infty$ is the orthonormal basis of $L^2(\Omega,dx)$
consisting of eigenfunctions of $L$ then $\{\breve{\phi}_k=U^{-1}\phi_k\}_{k=0}^\infty$ is an orthonormal basis of $L^2(\Omega, M(x)^2 dx)$.

Now given $u\in\Dom(L)$ we define $\breve{u}(x)=U^{-1}u(x)=M(x)^{-1}u(x)$, so that
$$u_{x_i}(x) = M(x) \breve{u}_{x_i}(x) + M_{x_i}(x)\breve{u}(x).$$
Therefore, for any $v\in\Dom(L)$,
  \begin{align*}
  \langle Lu,v \rangle  &= \int_{\Omega}\big(a^{ij}(x) u_{x_i} {{v_{x_j}}} +c(x) u {{v}}\big)\, dx\\ 
  &=\int_{\Omega}\bigg[ a^{ij}(x) \bigg( \breve{u}_{x_i}+\frac{M_{x_i}(x)}{M(x)}\breve{u}\bigg)
  \bigg({{\breve{v}_{x_j}}}+\frac{M_{x_j}(x)}{M(x)}{{\breve{v}}}\bigg)+ c(x) \breve{u}{{\breve{v}}}\bigg]M(x)^2\,dx.
  \end{align*}
This allows us to define the operator $\breve{L}$ in the following way.
For $\breve{u},\breve{v} \in  L^2(\Omega, M(x)^2 dx)$ such that $u=U(\breve{u})=M\cdot\breve{u}$
and $v=U(\breve{v})=M\cdot\breve{v}$ are in $\Dom(L)$, we define
$$\langle \breve{L} \breve{u}, \breve{v} \rangle:= \langle Lu, v \rangle.$$
With this, $(\lambda_k,\breve{\phi}_k)_{k=0}^\infty$ are the eigenvalues and eigenfunctions of $\breve{L}$,
where $\lambda_k$ are the eigenvalues of $L$.
Whence, 
$${\Dom(\breve{L})=}\Big\{\breve{u}\in L^2(\Omega, M(x)^2dx):\sum_{k=0}^\infty \lambda_k\breve{u}^2_k < \infty \Big\},$$
where $\displaystyle\breve{u}_k=\int_{\Omega}\breve{u}\breve{\phi}_kM(x)^2\,dx=\int_{\Omega}u\phi_k\,dx=u_k$. Observe that
$$\int_{\Omega} U(\breve{u})(x){{v(x)}}\,dx = \int_{\Omega} \breve{u}(x) U^{-1}{{v (x)}} M(x)^2\,dx.$$
Then we can formally write
$$\langle \breve{L}\breve{u},\breve{v} \rangle = \langle L(U\breve{u}), (U\breve{v}) \rangle
= \langle U^{-1}LU\breve{u},\breve{v} \rangle,$$
or
$$\breve{L} =U^{-1}\circ L\circ U.$$

\subsection{Composition of multiplication and change of variables}

We consider the following composition of the multiplication operator
$U$ with the change of variables operator $W$:
$$U \circ W : L^2(\widetilde{\Omega},M(h^{-1}(\widetilde{x}))^2J_{h^{-1}}d\widetilde{x}) \rightarrow L^2({\Omega},  dx ).$$
Notice that if $\bar{f}\in L^2(\widetilde{\Omega},M(h^{-1}(\widetilde{x}))^2J_{h^{-1}}d\widetilde{x})$ then
$$\int_{\Omega}|[(U\circ W)\bar{f}](x)|^2\,dx=\int_{\widetilde{\Omega}}|\bar{f}(\widetilde{x})|^2M(h^{-1}(\widetilde{x}))^2
J_{h^{-1}}(\tilde{x})\,d\widetilde{x}.$$
By using a similar technique as we used in cases of $W$ and $U$ separately, we can define
a new operator $\bar{L}$ in the following way.
For $\bar{u}, \bar{v} \in L^2(\widetilde{\Omega}, M(h^{-1}(\widetilde{x}))^2J_{h^{-1}}d\widetilde{x} )$
such that $u:=(U \circ W)\bar{u}$ and $v:=(U \circ W)\bar{v}$ are in $\Dom(L)$ we let
$$\langle \bar{L}\bar{u}, \bar{v} \rangle = \langle Lu, v \rangle.$$

By proceeding as in the previous cases we can formally write
$$\bar{L}=(U\circ W)^{-1}\circ L \circ (U\circ W).$$

\subsection{Transference method from $(\partial_t+L)^s$ to $(\partial_t+\bar{L})^s$}

Now we consider the parabolic operators $H= \partial_t + L$ and $\bar{H}=\partial_t+\bar{L}$,
where $L$ and $\bar{L}$ are as above. If $\bar{u}=\bar{u}(t,\widetilde{x})$
is a function of $t\in\R$ and $\widetilde{x}\in\widetilde{\Omega}$ then
the composition operator will act on $\bar{u}$ by leaving the variable $t$ fixed:
$$(U \circ W)\bar{u}(t,x)=M(x)\bar{u}(t,h(x)),\quad\hbox{for}~x\in\Omega,$$
so that
$$U \circ W :L^2(\R,dt; L^2(\widetilde{\Omega},M(h^{-1}(\widetilde{x}))^2J_{h^{-1}}d\widetilde{x}))
\rightarrow L^2(\R, dt; L^2({\Omega}, dx ))=L^2(\R\times\Omega).$$
Recall that
$$\Dom (H) = \Big\{ u \in L^2(\R \times \Omega): \int_{\R} \sum_{k=0}^\infty |(i \rho + \lambda_k)| |\widehat{u}_k(\rho)|^2\, d\rho<\infty\Big\}$$
and that, for $u\in \Dom(H)$ any $v\in C^\infty_c(\R\times\Omega)$,
$$\langle Hu,v \rangle_{L^2(\R\times\Omega)}=\int_{\R}\int_\Omega\Big(-u v_t+\sum_{i,j=1}^na^{ij}(x)u_{x_i}(t,x)v_{x_j}(t,x)+c(x) u(t,x) v(t,x)\Big) \,dx \, dt.$$
Now, for $\bar{u} \in L^2(\R,dt ; L^2(\widetilde{\Omega},M(h^{-1}(\widetilde{x}))^2J_{h^{-1}}d\widetilde{x}))$
such that $u:=(U \circ W)\bar{u}\in\Dom(H)$, and $v:=(U \circ W)\bar{v}$,
we define the parabolic operator
$$\langle \bar{H} \bar{u}, \bar{v} \rangle := \langle Hu,v \rangle.$$
As a matter of fact, we can write,
\begin{align*}
\langle Hu,v \rangle_{L^2(\R\times\Omega)}  &= \int_{\R}\int_\Omega\bigg[-M(x)\bar{u}(t,h(x)) M(x) \bar{v}_t(t,h(x))  \\ 
&\qquad\qquad +\sum_{i,j=1}^na^{ij}(x)\Big( M_{x_i}(x) \bar{u}(t,h(x)) + \sum_{k=1}^nM(x)\bar{u}_{\widetilde{x}_k}(t,h(x))
(\nabla h(x))_{ki}  \Big) \\
&\qquad\qquad\quad\times \Big( M_{x_j}(x) \bar{v}(t,h(x)) + \sum_{\ell=1}^nM(x)\bar{v}_{\widetilde{x}_\ell}(t,h(x))
(\nabla h(x))_{\ell j}\Big) \\
&\qquad\qquad+c(x) M(x)\bar{u}(t,h(x))M(x) \bar{v}(t,h(x))\bigg]\,dx\,dt \\
&= \int_{\R}\int_{\widetilde{\Omega}}\bigg[-\bar{u}\bar{v}_t  \\ 
&\qquad\qquad +\sum_{i,j=1}^na^{ij}(h^{-1}(\widetilde{x}))\bigg( \frac{M_{x_i}(h^{-1}(\widetilde{x}))}{M(h^{-1}(\widetilde{x}))}\bar{u} + \sum_{k=1}^n
\bar{u}_{\widetilde{x}_k}(\nabla h(h^{-1}(\widetilde{x})))_{ki}  \bigg) \\
&\qquad\qquad\quad\times \bigg( \frac{M_{x_j}(h^{-1}(\widetilde{x}))}{M(h^{-1}(\widetilde{x}))} \bar{v}(t,\widetilde{x})
+ \sum_{\ell=1}^n\bar{v}_{\widetilde{x}_\ell} (\nabla h(h^{-1}(\widetilde{x})))_{\ell j}\bigg) \\
&\qquad\qquad+c(h^{-1}(\widetilde{x})) \bar{u} \bar{v}\bigg]
M(h^{-1}(\widetilde{x}))^2J_{h^{-1}}\,d\widetilde{x}\,dt \\
&=  \langle \bar{H} \bar{u}, \bar{v} \rangle 
\end{align*}
By using a similar argument as before we can formally write
$$\bar{H} = (U \circ W)^{-1} \circ H \circ (U \circ W).$$

Next, for $u\in\Dom(H)$ set $\displaystyle u_k(t)=\int_{\Omega}u\phi_k\,dx$, and write
$$u(t,x) =\sum_{k=0}^\infty u_k(t) \phi_k(x).$$
We know from the previous discussion that $(\lambda_k,\bar{\phi}_k)_{k=0}^\infty$ is the family
of eigenvalues and eigenfunctions of $\bar{L}$, where
$$\bar{\phi}_k(\widetilde{x})=\frac{1}{M(h^{-1}(\widetilde{x}))} \phi_k(h^{-1}(\widetilde{x}))\quad\hbox{for}~x\in\widetilde{\Omega}.$$
So if $u(t,x)\in L^2(\R,dt; L^2(\widetilde{\Omega},M(h^{-1}(\widetilde{x}))^2J_{h^{-1}}d\widetilde{x}))$, then
$$\bar{u}(t,\widetilde{x}) =\sum_{k=0}^\infty\bar{u}_k(t)\frac{1}{M(h^{-1}(\widetilde{x}))} \phi_k(h^{-1}(\widetilde{x})).$$
But
$$\bar{u}_k(t)=\int_{\widetilde{\Omega}}\bar{u}(t,\widetilde{x})\bar{\phi}_k(\widetilde{x})
M^2(h^{-1}(\widetilde{x}))J_{h^{-1}}\,d\widetilde{x}=\int_\Omega u(t,x)\phi_k(x)\,dx=u_k(t).$$
Hence,
$$\langle \bar{H} \bar{u}, \bar{v}\rangle  = \langle Hu,v \rangle=
\int_{\R} \sum_{k=0}^\infty(i\rho + \lambda_k)\widehat{u}_k(\rho) \overline{\widehat{v}_k(\rho)}\,d \rho
=\int_{\R} \sum_{k=0}^\infty(i\rho + \lambda_k)\widehat{\bar{u}}_k(\rho) \overline{\widehat{\bar{v}}_k(\rho)}\,d \rho.$$
Therefore, for any $0\leq s\leq1$,
$$\langle \bar{H}^s \bar{u}, \bar{v}\rangle
=\int_{\R} \sum_{k=0}^\infty(i\rho + \lambda_k)^s\widehat{\bar{u}}_k(\rho) \overline{\widehat{\bar{v}}_k(\rho)}\,d \rho
=\langle H^su,v \rangle.$$
Whence, we can formally write
$$\bar{H}^s = (U \circ W)^{-1}\circ H^s\circ (U \circ W).$$

\begin{proof}[Proof of Theorem \ref{thm:transference}]
Let us first show how to transfer Theorem \ref{harn_para}.
Let $\bar{u} \in \Dom(\bar{H}^s)$ be a solution to
$$\begin{cases}
\bar{H}^s \bar{u} = 0 &\hbox{in}~(0,1)\times\widetilde{O} \\
\bar{u} \geq 0&\hbox{in}~(-\infty,1) \times \widetilde{\Omega},
\end{cases}$$
for some open set $\widetilde{O} \subset \widetilde{\Omega}$.
From the definition, $\langle \bar{H}^s\bar{u}, \bar{v} \rangle = \langle H^su,v \rangle$,
where $u = (U \circ W)\bar{u}$ and $ v = (U \circ W)\bar{v}$. Then, by taking any $v\in C^\infty_c((0,1)\times O)$,
where $O=h^{-1}(\tilde{O})$,
we can let $\bar{v}=(U\circ W)^{-1}v\in C_c^{\infty}((0,1) \times \widetilde{O})$ and thus conclude that
$H^s u = 0 $ in $(0,1) \times h^{-1}(\widetilde{O})=(0,1)\times O$. Also $u \geq 0$ in
$(-\infty,1) \times h^{-1}(\widetilde{\Omega})=(-\infty,1)\times\Omega$. 
Let $\widetilde{J}$ be a compact subset of $\widetilde{O}$. Then $h^{-1}(\widetilde{J})$
is a compact subset of $O$ and, by Harnack inequality for $H^s$, (see Remark \ref{Harnackint}),
$$\sup_{(\frac{1}{4},\frac{1}{2}) \times h^{-1}(\widetilde{J})} u \leq C \inf_{(\frac{3}{4},1) \times h^{-1}(\widetilde{J})}u.$$
Since $M(x)$ is strictly positive, continuous and bounded in $h^{-1}(\widetilde{J})$, 
$$\sup_{(\frac{1}{4},\frac{1}{2}) \times h^{-1}(\widetilde{J})} W \bar{u} \leq C' \inf_{(\frac{3}{4},1) \times h^{-1}(\widetilde{J})} W \bar{u}.$$
The change of variable $h$ is a smooth diffeomorphism, so that
$$\sup_{(\frac{1}{4},\frac{1}{2}) \times \widetilde{J}}  \bar{u} \leq C' \inf_{(\frac{3}{4},1) \times \widetilde{J}}  \bar{u}.$$
Thus Harnack inequality holds for $\bar{H}^s$. Let $\widetilde{K}$ be a compact subset of $(0,1)\times\widetilde{O}$.
Then $K=h^{-1}(\widetilde{K})$ is a compact subset of $(0,1)\times O$ and $u$ is parabolically H\"older continuous in $K$ with
$$\|u\|_{C^{\alpha/2,\alpha}_{t,x}(K)}\leq C\|u\|_{L^2(\R\times\Omega)}=
C\|\bar{u}\|_{L^2(\R,dt;L^2(\widetilde{\Omega},M(h^{-1}(\widetilde{x}))^2J_{h^{-1}}d\widetilde{x}))}.$$
Notice that $\bar{u}(t,\widetilde{x})=[(U \circ W)^{-1}u](t,\widetilde{x})=\frac{1}{M(h^{-1}(\widetilde{x}))}u(t,h^{-1}(\widetilde{x}))$, which,
for any $(t_i,x_i)=(t_i,h^{-1}(\widetilde{x}_i))\in K$, $i=1,2$, gives
\begin{align*}\nonumber
|\bar{u}(t_1,\widetilde{x}_1)-\bar{u}(t_2,\widetilde{x}_2)|&=\left| \frac{u(t_1,x_1)}{M(x_1)}- \frac{u(t_2,x_2)}{M(x_2)} \right| \\
& \leq  \left| \frac{u(t_1,x_1)}{M(x_1)}- \frac{u(t_1,x_1)}{M(x_2)} \right| + \left| \frac{u(t_1,x_1)}{M(x_2)}- \frac{u(t_2,x_2)}{M(x_2)} \right| \\
&\leq C \|M^{-1}\|_{C^\alpha_x(K)}\|u\|_{C^{\alpha/2,\alpha}_{t,x}(K)}d((t_1,x_1),(t_2,x_2))^\alpha \\
&\leq C'\|\bar{u}\|_{L^2(\R,dt;L^2(\widetilde{\Omega},M(h^{-1}(\widetilde{x}))^2J_{h^{-1}}d\widetilde{x}))}
d((t_1,\widetilde{x}_1),(t_2,\widetilde{x}_2))^\alpha
\end{align*}
where $d$ denotes the parabolic distance. In the last identity we used the fact that $h^{-1}$ is a smooth diffeomorphism.

Let us next transfer the boundary Harnack inequality of Theorem \ref{harn_bd}.
Again, for simplicity and without loss of generality, we consider $\tilde{x}=0$. Let $\bar{u} \in \Dom(\bar{H}^s)$ be a solution to
$$\begin{cases}\nonumber
\bar{H}^s\bar{u} = 0&\hbox{in}~(-2,2) \times(\widetilde{\Omega}_0\cap \widetilde{B}_{2r}(0)) \\
\bar{u}  \geq  0&\hbox{in}~(-\infty, 2) \times \widetilde{\Omega},
\end{cases}$$
such that $\bar{u}$ vanishes continuously on
$(-2,2) \times ( (\widetilde{\Omega}\setminus \widetilde{\Omega}_0) \cap \widetilde{B}_{2r}(0))$.
Let $(t_0, \widetilde{x}_0)$ be a fixed point in $(-2,2) \times \widetilde{\Omega}_0$
such that $t_0 > 1$.
Then $H^su = 0$ in $(-2,2) \times (\Omega_0\cap h^{-1}(\widetilde{B}_{2r}(0)))$,
where $\Omega_0=h^{-1}(\widetilde{\Omega}_0)$,
$u\geq0$ in $(-\infty,2)\times\Omega$ and,
as $h$ is a smooth diffeomorphism, we can also see that
$u = (U \circ W) \bar{u}$ vanishes continuously in
$(-2,2) \times((\Omega \setminus \Omega_0) \cap h^{-1}(\widetilde{B}_{2r}(0)))$.
We assume, again for simplicity, that $h(0)=0$ and let $K = h^{-1}(\widetilde{B}_{r}(0))$.
Then $0\in K$ and $K$ is compactly contained in $h^{-1}(\widetilde{B}_{2r}(0)))$.
We know that (see Remark \ref{Harnackbounda})  
$$\sup_{(-1,1) \times (\Omega_0 \cap K)} u(t,x) \leq C u(t_0,x_0),$$
for $C>0$. Since $M>0$ is bounded and continuous, and $h$ is a smooth diffeomorphism,
$$\sup_{(-1,1) \times (\widetilde{\Omega}_0 \cap\widetilde{B}_r(0))} \bar{u}(t,\widetilde{x}) \leq C' \bar{u}(t_0,\widetilde{x}_0).$$
\end{proof}

\begin{rem}
As it was explained in Remark \ref{rem:continuousspectrumextension}, one can check that
if the differential operator $L$ has continuous spectrum, then all the previous transference results are still valid.
\end{rem}

Now we use the transference method of Theorem \ref{thm:transference} to prove the following result.

\begin{thm}\label{thm:trans}
	Theorems \ref{harn_para} and \ref{harn_bd} hold true for solutions $u$ to $(\partial_t+L)^su=f$,
	where $L$ is any of the elliptic operators in $(7)$--$(10)$.
\end{thm}

\begin{proof}
We have already proven Theorem \ref{thm:trans} for the elliptic operators $L$ in $(2)$--$(6)$.

{\it Transference from $(2)$ to $(7)$.}
In this case, $H^s=(\partial_t-\Delta+|x|^2-n)^s$ in $\R\times\Omega=\R\times\R^n$ with Lebesgue measure
and with zero boundary condition at infinity whereas $\bar{H}^s=(\partial_t- \Delta + 2x\cdot\nabla)^s$
in $\R\times\widetilde{\Omega}=\R\times\R^n$ with Gaussian measure $\pi^{-n/4}e^{-|x|^2/2} dx$. For the transference
we use $h(x)=x$ and $M(x) = \pi^{-n/4}e^{-|x|^2/2}$.
 
{\it Transference from $(3)$ to $(8)$.}
In all these examples we have $\widetilde{\Omega}=\Omega$. In the first three cases
we start with $H^s=(\partial_t-\frac{1}{4}(\Delta + |x|^2 + \sum^n_{i=1} \frac{1}{x^2_i} \left( \alpha^2_i - \frac{1}{4} \right)))^s$,
for $\alpha_i>-1$, in $\R\times\Omega=\R\times(0, \infty)^n$. By using the transference method
we can obtain the result for the other Laguerre systems. 
 \begin{itemize}
 \item For $\bar{H}^s=(\partial_t + \sum^n_{i=1}(  -x_i \frac{\partial^2}{\partial x^2_i}  -(\alpha_i + 1) \frac{\partial}{\partial x_i} + \frac{x_i}{4}))^s$ with measure $ x^{\alpha_1}_1 \cdot \cdot \cdot x^{\alpha_n}_n  dx$,
 which is related to the Laguerre system $l^{\alpha}_k$, we choose $h(x) = (x^2_1, x^2_2,\ldots,x^2_n)$
 and $M(x) = 2^{n/2} x_1^{\alpha_1 + 1/2}\cdots x_n^{\alpha_n + 1/2}$.
 \item For $\bar{H}^s=(\partial_t + \frac{1}{4}(-\Delta + |x|^2) - \sum^n_i \frac{2 \alpha_i + 1}{4x_i} \frac{\partial}{\partial x_i})^s$
 with measure $x^{2 \alpha_1+1}_1 \cdot \cdot \cdot x^{2 \alpha_n + 1}_n dx$,
 which is related to the Laguerre system $\psi^{\alpha}_k$, we choose $h(x)=x$
 and $M(x) = x_1^{\alpha_1 + 1/2}\cdots x_n^{\alpha_n + 1/2}$. 
 \item For $\bar{H}^s=(\partial_t+\sum^n_{i=1}(-x_i \frac{\partial^2}{\partial x^2_i}-\frac{\partial}{\partial x_i}+\frac{x_i}{4}+\frac{\alpha^2_i}{4x_i}))^s$ with Lebesgue measure,
 which is related to the Laguerre system $\mathcal{L}^{\alpha}_k$, we choose $h(x) = (x^2_1, x^2_2,\ldots,x^2_n)$
 and $M(x) = 2^{n/2} x_1^{ 1/2}\cdots x_n^{ 1/2}$.
 \end{itemize}
In the last case, we start with $H^s=(\partial_t-\frac{1}{4}(\Delta + |x|^2 + \sum^n_{i=1} \frac{1}{x^2_i} \left( \alpha^2_i - \frac{1}{4} \right))-\frac{\alpha+1}{2})^s$. Thus, we apply the transference method for $\bar{H}^s=(\partial_t + \sum^n_{i=1}( - x_i \frac{\partial^2}{\partial x^2_i} - (\alpha_i + 1 - x_i) \frac{\partial}{\partial x_i}))^s$
 with measure $x_1^{\alpha_1}e^{-x_1}\cdots x_ne^{-x_n}dx$,
 which is related to the Laguerre polynomials system $L^{\alpha}_k$, by choosing $h(x) = (x^2_1, x^2_2,\ldots,x^2_n)$
 and $M(x) = 2^{n/2} e^{- |x|^2/2} x_1^{\alpha_1 + 1/2}\cdots x_n^{\alpha_n + 1/2}$. 
 
{\it Transference from $(4)$ to $(9)$.}
In this case, $H^s=(\partial_t - \frac{d^2}{dx^2}  + \frac{\lambda(\lambda-1)}{\sin^2x})^s$
in $\R\times\Omega=\R\times(0,\pi)$ with Lebesgue measure,
and $\bar{H}^s=(\partial_t- \frac{d^2}{dx^2} - 2 \lambda \cot x \frac{d}{dx} + \lambda^2)^s$
in $\R\times\widetilde{\Omega}=\R\times(0,\pi)$ with measure $\sin^{2\lambda}x dx$. For the transference method we use $h(x)=x$
and $M(x) =(\sin x)^\lambda$.
  
{\it Transference from $(6)$ to $(10)$.}
Here $\Omega=\widetilde{\Omega}=(0,\infty)$,
$H^s=(\partial_t - \frac{d^2}{dx^2} + \frac{\lambda^2-\lambda}{x^2})^s$
in $\R\times(0,\infty)$ with Lebesgue measure and $\bar{H}^s=(\partial_t-\frac{d^2}{dx^2} - \frac{2 \lambda}{x} \frac{d}{dx})^s$
in $\R\times(0,\infty)$ with measure $x^{2 \lambda} dx $. For the transference method we use $h(x)=x$
and $M(x) =x^\lambda$.
\end{proof}

\bigskip

\noindent\textbf{Acknowlegments.}~We are grateful to the referees for their detailed reading of the paper
and many useful suggestions that helped us to improve the presentation.

\medskip



\end{document}